\documentclass{article}
\usepackage{graphicx} 
\usepackage{mathtools}
\usepackage[shortlabels]{enumitem} 
\usepackage{amsfonts}
\usepackage{amssymb}
\usepackage[left=1in]{geometry}
\usepackage{geometry}
\usepackage{wrapfig}
\usepackage{setspace}
\usepackage{epsfig}
\usepackage{color}
\usepackage{tikz}
\usepackage{float}
\usepackage{url}
\usepackage[hidelinks]{hyperref}
\usepackage{pst-plot}
\usepackage{verbatim}
\usepackage{epic}

\usepackage{latexsym}
\usepackage{amssymb}
\usepackage{amsthm}
\usepackage{amscd}
\usepackage{amsmath}
\usepackage{tikz-cd}
\usepackage{soul}
\usepackage{subfiles}

\newtheorem{theorem}{Theorem}[section]
\newtheorem{fact}[theorem]{Fact}

\newtheorem{lemma}[theorem]{Lemma}

\newtheorem{conjecture}[theorem]{Conjecture}

\theoremstyle{definition}
\newtheorem{definition}[theorem]{Definition}
\newtheorem{example}[theorem]{Example}
\newtheorem{remark}[theorem]{Remark}

\newtheorem{case}{Case}

\numberwithin{equation}{section}
\numberwithin{case}{theorem}

\newcommand{\cP}{\mathcal{P}}		

\newcommand{\bbN}{\mathbb{N}}		
\newcommand{\bbP}{\mathbb{P}}		
\newcommand{\bbQ}{\mathbb{Q}}		
\newcommand{\bbZ}{\mathbb{Z}}		

\renewcommand{\div}{\operatorname{div}} 
\newcommand{\ord}{\operatorname{ord}}   
\newcommand{\MD}{\operatorname{MD}} 	

\newcommand{\union}{\cup}

\newcommand{\cross}{\times}

\newcommand{\floor}[1]{\lfloor #1 \rfloor}
\newcommand{\ceil}[1]{\lceil #1 \rceil}
\newcommand{\Floor}[1]{\left\lfloor #1 \right\rfloor}
\newcommand{\Ceil}[1]{\left\lceil #1 \right\rceil}
\newcommand{\size}[1]{\lvert #1 \rvert}

\newcommand{\textand}{\quad \text{and} \quad}

\renewcommand{\>}{\right\rangle}
\newcommand{\ignore}[1]{}

\newcommand{\ib}{{\mathrm{(b)}}}
\newcommand{\ic}{{\mathrm{(c)}}}
\newcommand{\id}{{\mathrm{(d)}}}

\newcommand{\solid}[1]{\put#1{\circle*{0.25}}}
\newcommand{\open}[1]{\put#1{\circle{0.25}}}
\newcommand{\xdot}[1]{\put#1{\makebox(0,0){\small +}}}

\title{Section Rings of $\bbQ$-Divisors on Genus $1$ Curves}
\author{Michael Cerchia, Jesse Franklin, Evan O'Dorney}
\date{2024}

\psset{unit=0.5cm,linewidth=0.05}

\begin{document}
    
    \maketitle
    
    \begin{abstract}
        We compute generators and relations for the section ring of a rational divisor on an elliptic curve. Our technique generalizes the work of O'Dorney (in genus zero) and Voight--Zureick-Brown (for specific divisors arising from the study of stacky curves). For effective divisors supported on at most two points, we give explicit descriptions of the generators and the leading terms of the relations for a minimal presentation. As in the genus zero case, the generators are parametrized by best lower approximations to the coefficients, but there are added wrinkles. Following Landesman, Ruhm and Zhang we can bound the degrees of generators for the section ring of an effective divisor supported at any finite number of points.
    \end{abstract}
    
    \tableofcontents
    
    \newpage
    
    \section{Motivation}
    Let $C$ be an algebraic curve (smooth, projective) over a field $\Bbbk$, which we will take to be algebraically closed for convenience, although this assumption is not necessary. Let $D$ be a $\bbQ$-divisor on $C$. We denote by $H^0(D)$ (short for $H^0(C,D)$) the $\Bbbk$-vector space of rational functions $f$ on $C$ such that $\div f + D$ is effective. Our aim in this paper is to understand the generators and relations of the \textbf{section ring}
    \[
      S_D = \bigoplus_{d \geq 0} H^0(dD),
    \]
    where the ring structure comes from the multiplication map $\cdot \colon H^0(dD) \cross H^0(eD) \to H^0((d+e) D)$. We will find it convenient to introduce a bookkeeping variable $u$ and write
    \[
      S_D = \bigoplus_{d \geq 0} u^d H^0(dD).
    \]
    To exclude degenerate cases, we assume that $D$ is ample (that is, $\deg D > 0$). 
    
    Section rings, most notably the \textbf{canonical ring} where the divisor $D = K_C$ is a canonical divisor, are widely used to produce embeddings of a curve into (weighted) projective spaces. With a theory of general $\bbQ$-divisors, one can compute log canonical rings of stacky curves $C$, as in \cite{VZB}, where the support of the \textbf{log divisor} (the cusps of $C$) consists of stacky points. This is useful to number theorists who consider log canonical rings of modular curves, such as in the Drinfeld setting \cite{Franklin-geometry-Drinfeld-modular-forms}. Another application of our generalization is the case of non-tame stacky curves, such as \cite[Remark $5.3.11$]{VZB} where the characteristic of the ground field divides the order of the stabilizer of some point on the curve. In particular, a work in progress by Kobin and Zureick-Brown studies modular forms mod $p = 2$ and $3$, exploiting the phenomenon that elliptic curves with $j = 0$ have extra automorphisms in characteristic $2$ or $3$. In these characteristics, modular curves are \emph{wild} and their log-canonical rings are not addressed by \cite{VZB}; Kobin and Zureick-Brown rely very heavily on the explicit computations of \cite{O'Dorney} and of the present paper. Additionally, our paper ``finishes'' this line of inquiry, in the sense that we do not expect there to be any nice answer to this type of question for small-degree divisors on higher-genus stacky curves.
    
    Our work is organized similarly to \cite{O'Dorney}. First, we discuss the case of a divisor $D = \alpha P$ supported at one point. We find that most of the generators of the section ring $S_D$ are indexed by the best lower approximations to $\alpha$, just as in the genus $0$ case; but there can be up to three exceptional generators depending on whether the fractional part $\{-1/\alpha\}$ lies within various subintervals ($[0,1/3)$, $[1/3,1/2)$, etc.) within the unit interval $[0,1)$ (Section \ref{sec:1point}).
    
    Next we consider the two-point effective case $D = \alpha^{(1)} P^{(1)} + \alpha^{(2)} P^{(2)}$, $\alpha^{(k)} > 0$. Here the section ring combines the flavor of the one-point cases from genera $0$ and $1$ (Section \ref{section: 2pt effective}).
    
    We would also like to address ineffective divisors (such as $\alpha^{(1)} P^{(1)} - \alpha^{(2)} P^{(2)}$) and divisors supported at more than two points. However, a complete description of the section ring is more elusive in these cases, and at the moment we merely offer interesting examples and conjectures (Section \ref{sec:unans}). Thanks to \cite{Landesman-Ruhm-Zhang-Spin-canonical-rings} there are upper bounds for the degrees of generators and relations for a large family of $\bbQ$-divisors supported at an arbitrary number of points on curves of all genera. However, not every possible $\bbQ$-divisor is covered by the existing theory of \cite{VZB} and \cite{Landesman-Ruhm-Zhang-Spin-canonical-rings}, hence our work.
    
    \subsection{Main results} For our readers' convenience we state our main theorems here, slightly simplified for readability.
    
    For comparison, here is the result of the last author from \cite{O'Dorney} concerning the one-point case on $\bbP^1$:
    \begin{theorem}[\cite{O'Dorney}, Theorem 4]%
        \label{thm:1point_P1} 
        Let $D = \alpha (\infty)$ be a $\bbQ$-divisor supported on one point of $\bbP^1$. Denote by $t$ a coordinate on $\bbP^1$ having a simple pole at this point. Let
        \[
        0 = \frac{c_0}{d_0} < \frac{c_1}{d_1} < \cdots < \frac{c_r}{d_r} = \alpha
        \]
        be the nonnegative best lower approximations to $\alpha$. Then $S_D$ has a minimal presentation consisting of the $r+1$ generators $f_i = t^{c_i}u^{d_i}$ and $\binom{r}{2}$ relations of the form
        \begin{equation} \label{eq:1ptrel_P1}
            g_{ij} = f_i f_j - f_{h_{ij}}^{a_{ij}} \; (i < h_{ij} < j) \quad \text{or} \quad g_{ij} = f_i f_j - f_{h_{ij}}^{a_{ij}} f_{h_{ij}+1}^{b_{ij}} \; (i < h_{ij} < h_{ij}+1 < j)
        \end{equation}
        for each $(i,j)$ with $j \geq i+2$, with some positive integers $a_{ij}$ and $b_{ij}$.
    \end{theorem}

    In general, we show that the section ring on an elliptic curve $C$ has a very similar form to that on $\bbP^1$, with necessary modifications owing to the issue that there is no rational function of degree $1$ on $C$. Unlike the genus $0$ case, where the relations are explicit binomials, here the relations have lower-order terms depending in a complicated way on the particular elliptic curve $C$ chosen. Hence we content ourselves with determining the \emph{leading terms,} and in particular the degrees, of the relations.
    
    \begin{theorem}[see Theorems \ref{thm:1point_gens}, \ref{thm:1point_rels}] \label{thm:1point_intro}
        Let $C$ be an elliptic curve with a marked point $\infty$, and denote by $t_c$ ($c \in \bbZ_{\geq 0}, c \neq 1$) a function whose polar divisor is $c(\infty)$. For $D = \alpha(\infty)$ a $\bbQ$-divisor supported at this point, let
        \[
        0 = \frac{c_0}{d_0} < \frac{c_1}{d_1} < \cdots < \frac{c_r}{d_r} = \alpha
        \]
        be the nonnegative best approximations to $\alpha$. Then:
        \begin{enumerate}[$($a$)$]
            \item $S_D$ has a minimal generating set consisting of the functions 
            \begin{itemize}
                \item $f_i = t_{c_i} u^{d_i}$ for $i \neq 1$ (observe that $c_1/d_1$ is always the unique best lower approximation with numerator $1$, and therefore inadmissible here);
                \item $f_\ib = t_2 u^{\ceil{2/\alpha}}$ if $\{-1/\alpha\} \in [0,1/2)$;
                \item $f_\ic = t_3 u^{\ceil{3/\alpha}}$ if $\{-1/\alpha\} \in [0,1/3)$;
                \item $f_\id = t_{c_1 + c_2} u^{d_1 + d_2}$ if $\{-1/\alpha\} \in (0,1/2)$.
            \end{itemize}
            \item With respect to a suitable term ordering, a Gr\"obner basis for the relations of $S_D$ consists of relations with the following leading terms:
            \begin{itemize}
                \item All products $f_i f_j$, where $3 \leq i \leq r, 0 \leq j \leq i - 2, j \neq 1$, except possibly $f_3 f_0$;
                \item All products $f_i f_\ib$, $f_i f_\ic$, and $f_i f_\id$, where $3 \leq i \leq r$, if the respective generator $f_\ib$, $f_\ic$, $f_\id$ exists;
                \item At most six additional relations, according to whether $\{-1/\alpha\}$ lies in various intervals.
            \end{itemize}
            \item Moreover, the Gr\"obner basis is also minimal, except for at most one relation whose leading term has degree $3$ or $4$, the remaining relations having quadratic leading terms.
        \end{enumerate} 
    \end{theorem}
    
    Our next theorem generalizes this to a two-point effective divisor:
    \begin{theorem}[see Theorems \ref{thm:2point_gens}, \ref{thm:2point_rels_unequal}, \ref{thm:2point_rels_equal}] \label{thm:2point_intro}
        Let $D = \alpha^{(1)}(P^{(1)}) + \alpha^{(2)}(P^{(2)})$ be an effective $\bbQ$-divisor on an elliptic curve $C$ supported on two points $P^{(k)}$, with $\alpha^{(1)} \geq \alpha^{(2)}$. Let 
        \[
        0 = \frac{c_0^{(k)}}{d_0^{(k)}} < \frac{c_1^{(k)}}{d_1^{(k)}} < \cdots < \frac{c_{r^{(k)}}^{(k)}}{d_{r^{(k)}}^{(k)}} = \alpha^{(k)}
        \]
        be the best lower approximations to $\alpha^{(k)}$. Then:
        \begin{enumerate}[$($a$)$]
            \item $S_D$ has a minimal system of generators of the following forms:
            \begin{itemize}
                \item $f_i^{(k)} = t_{c_i^{(k)}}^{(k)} u^{d_i^{(k)}}$ for $k \in \{1, 2\}$ and $i = 0, 2, 3, \ldots, r^{(k)}$ (including $f_0 = f_0^{(1)} = f_0^{(2)} = u$);
                \item $f_\ib = t_{2}^{(1)} u^{\ceil{2/\alpha^{(1)}}}$ if $\{-1/\alpha^{(1)}\} \in [0,1/2)$;
                \item $f_\ic = t_{3}^{(1)} u^{\ceil{3/\alpha^{(1)}}}$ if $\{-1/\alpha^{(1)}\} \in [0,1/3) \union [1/2, 2/3)$ and $\ceil{1/\alpha^{(2)}} > \ceil{1/\alpha^{(1)}}$;
                \item $f_\id = t_{c_1^{(1)} + c_2^{(2)}}^{(1)} u^{d_1^{(1)} + d_2^{(1)}}$, if $\{-1/\alpha^{(1)}\} \in (0,1/2)$ and $\ceil{1/\alpha^{(2)}} > \ceil{1/\alpha^{(1)}}$;
                \item $f_w = w u^{d_1^{(2)}}$.
            \end{itemize}
            \item With respect to a suitable term ordering, a Gr\"obner basis for the relations of $S_D$ consists of relations with the following leading terms:
            \begin{itemize}
                \item All products $f_i^{(k)} f_j^{(k)}$, where $k \in \{1,2\}$, $3 \leq i \leq r^{(k)}, 0 \leq j \leq i - 2, j \neq 1$, except possibly $f_3^{(1)} f_0^{(1)}$;
                \item All products $f_w f_i^{(k)}$, where $i \geq 3$, $k \in \{1,2\}$;
                \item All products $f_i f_\ib$, $f_i f_\ic$, and $f_i f_\id$, where $k \in \{1,2\}$, $i \geq 4 - k$, if the respective generator $f_\ib$, $f_\ic$, $f_\id$ exists;
                \item At most nine additional relations, according to whether $\{-1/\alpha^{(1)}\}$ lies in various intervals and whether $\ceil{1/\alpha^{(1)}} = \ceil{1/\alpha^{(2)}}$ or not.
            \end{itemize}
            \item Moreover, the Gr\"obner basis is also minimal, except for at most one relation whose leading term has degree $3$ or $4$, the remaining relations having quadratic leading terms.
        \end{enumerate}
    \end{theorem}

    Finally, as a corollary to Lemma 4.4(c) of \cite{Landesman-Ruhm-Zhang-Spin-canonical-rings}, we obtain a bound on the generator and relation degrees for arbitrary effective $\bbQ$-divisors:
    \begin{theorem} \label{thm:bounds_intro}
      Let
      \[
        D = \sum_{i=1}^n \alpha^{(i)}(P^{(i)})
      \]
      be an effective divisor on a genus $1$ curve $C$, with the coefficients $\alpha^{(i)} = a^{(i)}/b^{(i)}$ in reduced form and the distinct points $P^{(i)}$ ordered so that
      \[
        \alpha^{(1)} \geq \cdots \geq \alpha^{(n)}.
      \]
      Then the section ring $S_D$ is generated in degrees at most
      \[
        B = \max\{3b^{(1)}, b^{(2)}, \ldots, b^{(n)}\},
      \]
      with relations in degrees at most $2B$.
    \end{theorem}
    These bounds are achievable: see Examples \ref{ex:Wei} and \ref{ex:3point}.
    
    \subsection{Acknowledgements} This project grew out of the AMS Mathematical Research Communities: Explicit Computations with Stacks conference in June 2023. We thank the organizers, and we further thank John Voight, Robin Zhang and David Zureick-Brown for helpful conversations.
    
    \section{The one-point case}
    \label{sec:1point}
    
    Fix an elliptic curve $C$ with a marked point $\infty$. We denote by $t_i$ a function on $C$ whose polar divisor is $i(\infty)$. We recall that $t_i$ exists for $i \in \bbZ_{\geq 2} \union \{0\}$; if $C$ is given by a Weierstrass equation $y^2 + a_1xy + a_3y = x^3 + a_2x^2 + a_4x + a_6$, then we may take
    \[
    t_i = \begin{cases}
        x^{i/2}, & \text{$i$ even} \\
        x^{(i-3)/2}y, & \text{$i$ odd.}
    \end{cases}
    \]
    
    In this section, we take a divisor $D = \alpha(\infty)$ and study the generators and relations of the resulting section ring $S_D$. Observe that there must be at least three generators and one relation (if $S_D$ were freely generated, it would yield a birational isomorphism of $C$ to some $\bbP^n$, which is impossible).
    
    \begin{example}\label{ex:Wei}
        Let $D = (\infty)$ consist of a single point with multiplicity $1$. Then $S_D$ has generators $u$, $x = u^2 t_2$, $y = u^3 t_3$ in degrees $1$, $2$, and $3$, respectively, and a single degree $6$ relation
        \[
        y^2 + a_1 u x y + a_3 u^3 y = x^3 + a_2 u^2 x^2 + a_4 u^4 x + a_6 u^6,
        \]
        a homogenization of the usual Weierstrass equation of the elliptic curve $C$. These generators are shown diagrammatically in Figure \ref{fig:Wei}, where we plot degree on the horizontal axis and pole order on the vertical axis. We use bullets for generators, open dots for other elements of $S_D$, and $+$'s to emphasize the nonexistence of elements in $S_D$ having a simple pole at $\infty$.
    \end{example}
    
    \begin{figure}[bhtp]
        \centering
        \begin{picture} (4.8,4.8) (-0.3,-0.3) 
            
            \drawline(0,0)(4.8,0)
            \drawline(0,0)(0,4.8)
            
            \drawline(0,0)(4.5,4.5)
            
            \solid{(1,0)}\solid{(2,2)}\solid{(3,3)}
            \put(1.2,0.1){\makebox(0,0)[bl]{\small $u$}}
            \put(2.2,1.9){\makebox(0,0)[ul]{\small $x$}}
            \put(3.2,2.9){\makebox(0,0)[ul]{\small $y$}}
            \xdot{(1,1)}\xdot{(2,1)}\xdot{(3,1)}\xdot{(4,1)}
            \open{(0,0)}
            \open{(2,0)}
            \open{(3,0)}\open{(3,2)}
            \open{(4,0)}\open{(4,2)}\open{(4,3)}\open{(4,4)}
        \end{picture}
        \caption{The section ring of $D = (P)$, which has three generators}
        \label{fig:Wei}
    \end{figure}
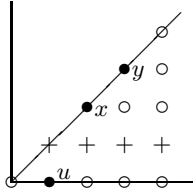
    
    In the following, we use without comment the following well-known characterization of the principal divisors on an elliptic curve $C$:
    \begin{fact}(\cite[Corollary $\mathrm{III}.3.5$]{Silverman-arithmetic-elliptic-curves})
        \label{fact: principal divisors on elliptic curves}
        A divisor $D=\sum n_P(P)$ on an elliptic curve is a principal divisor if and only if $\sum n_P=0$ (as integers) and $\sum (n_P P)=0$ (in the elliptic curve group law).
    \end{fact}
    
    \subsection{Generators}
    \begin{theorem} \label{thm:1point_gens}
        Let $D = \alpha(\infty)$ be a $\bbQ$-divisor on an elliptic curve $C$ supported at a single point $\infty$. Let
        \[
        0 = \frac{c_0}{d_0} < \frac{c_1}{d_1} < \cdots < \frac{c_r}{d_r} = \alpha
        \]
        be the nonnegative best approximations to $\alpha$. Then $S_D$ has a minimal generating set consisting of functions $f = t_c u^d$ for the following pairs $(d, c)$:
        \begin{enumerate}[$($a$)$]
            \item
            \label{it:best} $(d, c) = (d_i, c_i)$ for $i \neq 1$ (observe that $c_1/d_1$ is always the unique best lower approximation with numerator $1$, and therefore inadmissible here);
            \item
            \label{it:ord2} $(d, c) = (d_\ib, c_\ib) = (\ceil{2/\alpha}, 2)$ if $\{-1/\alpha\} \in [0,1/2)$; 
            \item
            \label{it:ord3} $(d, c) = (d_\ic, c_\ic) = (\ceil{3/\alpha}, 3)$ if $\{-1/\alpha\} \in [0,1/3) \union [1/2,2/3)$; 
            \item
            \label{it:c1+c2} $(d, c) = (d_\id, c_\id) = (d_1 + d_2, c_1 + c_2)$ if $\{-1/\alpha\} \in (0,1/2)$. 
        \end{enumerate}
        We denote each relevant lattice point by $v_i = (d_i, c_i)$, and the corresponding generator of the section ring by $f_i = t_{c_i} u^{d_i}$, where $i$ ranges over an index set $I$ containing $\{0, 2, 3,\ldots, r\}$ as well as the special symbols \ref{it:ord2}, \ref{it:ord3}, and \ref{it:c1+c2} in the cases in which they appear. 
    \end{theorem}
    
    \begin{proof}
        First, we transform the problem to finding generators for a certain semigroup. Observe that a $\Bbbk$-basis for $S_D$ is given by
        \begin{equation} \label{eq:M_basis}
            \{t_c u^d : (d,c) \in M\}
        \end{equation}
        where $M$ is the monoid
        \[
        M = \{(d,c) \in \bbZ^2 : 0 \leq c \leq \alpha d, c \neq 1\}.
        \]
        For $v = (d,c) \in M$ a vector, let $f_v = t_c u^d$ be the corresponding element of $S_D$. We cannot construct an isomorphism of $S_D$ with the monoid ring $\Bbbk[M]$ in this way, but the objects are closely related, and we will use the combinatorial structure of $M$ to probe the algebraic structure of $S_D$.
        
        Note that, owing to the grading by $d$, $M$ is an \textbf{atomic} monoid, that is, every element is a (not necessarily unique) sum of irreducibles. Consequently, $M$ has a unique minimal generating set, namely the irreducibles. Suppose that the following combinatorial lemmas about $M$ are proved:
        
        \begin{lemma}\label{lem:irred_M}
            The irreducibles of $M$ are exactly the pairs $(d,c)$ in the statement of the theorem.
        \end{lemma}
        
        \begin{lemma}\label{lem:unique}
            Let $(d,c)$ be an irreducible of $M$. Then any element $(d,c') \in M$ with $c' > c$ has a unique atomic decomposition.
        \end{lemma}
        
        Let us show that these two lemmas imply the statement of the theorem. Let $\cP$ be the set of irreducibles of $M$. Since $\cP$ generates $M$, the corresponding generating set $\{f_v : v\in \cP\}$ generates a subring $S' \subseteq S_D$ containing elements $t^{(d)}_c u^d$ for all $(d,c) \in M$ (where $t^{(d)}_c$ is a function on $C$ with a pole of order $c$ at $\infty$, possibly depending on $d$). These elements span $S_D$ as a $\Bbbk$-vector space, so $S' = S_D$.
        
        Now we show that each generator $f_v = t_c u^d$ is necessary. Let $S' \subseteq S_D$ be the subring generated by the $f_{v'}$, $v' \neq v$, and suppose for the sake of contradiction that $f_v \in S'$. Write
        \[
        f_v = a_1 f_1 + \cdots + a_k f_k,
        \]
        where $a_i \in \Bbbk$ and the $f_i \in S'_{\deg = d}$ are distinct products of the generators of $S'$. Since $v \in M$ is irreducible, no $f_i$ can have a pole of order exactly $c$, so two of them, say $f_1$ and $f_2$, must have a common larger order $c' > c$ to cancel the poles out. But by Lemma \ref{lem:unique}, this is impossible.
    \end{proof}
    
    \begin{proof}[Proof of Lemma \ref{lem:irred_M}]
        To understand the structure of $M$, we compare it to the simpler monoid
        \[
        M_0 = \{(d,c) \in \bbZ^2 : 0 \leq c \leq \alpha d\}
        \]
        in which the condition $c \neq 1$ has been omitted. This monoid controls the structure of the section ring for the corresponding situation in genus zero, and in the course of proving \cite[Theorem 4]{O'Dorney}, it was shown that the irreducibles of $M_0$ are precisely the vectors $(d_i,c_i)$ determined by the best lower approximations $c_i/d_i$.
        
        Since $M \subset M_0$, all such vectors remain irreducible in $M$ if they lie in $M$. Thus the vectors of type \ref{it:best} in Theorem \ref{thm:1point_gens} are irreducible. For types \ref{it:ord2} and \ref{it:ord3}, note that these are the simplest vectors in $M$ with $c$-coordinate $2$ and $3$, respectively, and cannot be decomposed, since $M$ has no elements with $c$-coordinate $1$. Thus they must be added unless they already appeared in type \ref{it:best}. For type \ref{it:ord2}, we have that $c/d = 2/\ceil{2/\alpha}$ is a best lower approximation (necessarily the second one $c_2/d_2$) if and only if
        \begin{align}
            \frac{2}{\ceil{2/\alpha}} &> \frac{1}{\ceil{1/\alpha}} \nonumber \\
            2 \Ceil{\frac{1}{\alpha}} &> \Ceil{\frac{2}{\alpha}} \nonumber \\
            \frac{2}{\alpha} + 2\left\{-\frac{1}{\alpha}\right\} &>
            \frac{2}{\alpha} + \left\{-\frac{2}{\alpha}\right\}\nonumber \\
            2\left\{-\frac{1}{\alpha}\right\} &> \left\{-\frac{2}{\alpha}\right\}. \label{eq:2_frac_parts}
        \end{align}
        Since
        \[
        \{2x\} = \begin{cases}
            2\{x\} & \{x\} < 1/2 \\
            2\{x\} - 1 & \{x\} \geq 1/2,
        \end{cases}
        \]
        the inequality \eqref{eq:2_frac_parts} holds exactly when $\{-1/\alpha\} \geq 1/2$, so the generator of type \ref{it:ord2} is needed whenever $\{-1/\alpha\} < 1/2$. For type \ref{it:ord3}, an analogous computation shows that $\{-1/\alpha\}$ must lie in the range $[0,1/3) \union [1/2,2/3)$ for $3/\ceil{3/\alpha}$ not to have already appeared as a best lower approximation.
        
        Finally, for type \ref{it:c1+c2}, note that if $-1/\alpha$ is an integer, then there is no $c_2/d_2$ because $c_1/d_1 = \alpha$ is the last approximation; while if $\{-1/\alpha\} \geq -1/2$, then $c_2 = 2$ as we found above, so $c_1 + c_2 = 3$, which pole order was already covered in types \ref{it:best} and \ref{it:ord3}. So we only need to consider type \ref{it:c1+c2} in the case $\{-1/\alpha\} \in (0,1/2)$. Here $c_2 \geq 3$ and $c_3$ (if it exists) is greater than $1 + c_2$, so the only irreducibles that could possibly appear in a decomposition of $v = (d_1 + d_2, c_1 + c_2)$ are
        \[
        (d_0, c_0) = (1,0), \quad (\ceil{2/\alpha}, 2), \quad (\ceil{3/\alpha}, 3), \quad (d_2, c_2).
        \]
        The last generator $(d_2, c_2)$ may be eliminated immediately since the difference $v - (d_2, c_2) = (d_1, 1)$ lies outside $M$. That leaves three generators lying in the submonoid
        \[
        \angle v_0v_1 = \langle(1,0), (d_1,1)\rangle = \langle(d_0,c_0), (d_1,c_1)\rangle
        = \left\{(d,c) \in \bbZ^2 : 0 \leq c \leq \frac{c_1}{d_1} d \right\}
        \]
        of $M_0$ determined by the first two best lower approximations of $\alpha$. But $v$ lies outside $\angle v_0v_1$ since $c_2/d_2 > c_1/d_1$, so $v$ is irreducible in $M$. This completes the proof that the claimed generators are irreducible and distinct.
        
        It remains to prove that there are no other irreducibles, that is, any nonzero vector $v \in M$ not among the ones listed is reducible in $M$. As an element of $M_0$, any $v$ lies in some angle $\angle v_i v_{i+1}$ and so can be decomposed as a positive integer linear combination
        \[
        v = a(d_{i}, c_{i}) + b(d_{i+1}, c_{i+1})
        \]
        of two consecutive generators of $M_0$. If $i \geq 2$, then these are also generators of $M$, so we only need to consider two cases:
        
        \begin{case}
            $i = 0$. Then $(d_0, c_0) = (1,0)$ is already a generator of $M$. We must have $b \neq 1$ since $v \in M$, so $b$ can be written as a sum of $2$'s and $3$'s, which yields an expression for $v$ in terms of the generators of types \ref{it:best}, \ref{it:ord2}, and \ref{it:ord3}.
        \end{case}
        
        \begin{case}
            $i = 1$, so
            \begin{equation} \label{eq:i=1}
                v = a(d_1, c_1) + b(d_2, c_2).
            \end{equation}
            We may assume that $a$ and $b$ are nonzero, or else we could have taken $i = 0$ or $i = 2$ respectively (in the latter case, allowing a zero coefficient on a possibly nonexistent $(d_3, c_3)$). If $a \neq 1$, note that each term individually belongs to $M$, so $v$ is reducible. So $a = 1$ and $b \geq 1$, and 
            \[
            v = (d_1 + d_2, c_1 + c_2) + (b-1)(d_2, c_2)    
            \]
            is either reducible or a generator of type \ref{it:c1+c2}. \qedhere
        \end{case}
    \end{proof}
    \begin{proof}[Proof of Lemma \ref{lem:unique}]
        Now let $v = (d,c)$ be an irreducible of $M$. We wish to prove that any element $v' = (d,c') \in M$ lying above $v$ has a unique atomic decomposition. By the previous lemma, $v$ is of one of the four types in Theorem \ref{thm:1point_gens}; we handle each type in turn. In Figure \ref{fig:unique}, we illustrate the various cases that can occur.
        
        \begin{figure}[bhtp]
            \centering
            \begin{picture}(4.5,7.7)(-0.3,-1)
                \drawline(0,0)(4.5,0)
                \drawline(0,0)(0,6.3)
                
                \drawline(0,0)(4.5,6.3)
                
                \solid{(1,0)}\solid{(2,2)}\solid{(3,3)}\solid{(3,4)}\solid{(4,5)}
                \put(1.2,0.1){\makebox(0,0)[bl]{\small $a$}}
                \put(2.2,1.9){\makebox(0,0)[ul]{\small $b$}}
                \put(3.2,2.9){\makebox(0,0)[ul]{\small $c$}}
                \put(3.2,3.9){\makebox(0,0)[ul]{\small $a$}}
                \put(4.2,4.9){\makebox(0,0)[ul]{\small $d$}}
                
                \xdot{(1,1)}\xdot{(2,1)}\xdot{(3,1)}\xdot{(4,1)}
                
                \open{(0,0)}
                \open{(2,0)}
                \open{(3,0)}\open{(3,2)}
                \open{(4,0)}\open{(4,2)}\open{(4,3)}\open{(4,4)}
                
                \put(0,-1){\makebox(4.5,0){$4/3 \leq \alpha < 3/2$}}
            \end{picture}\quad
            \begin{picture}(4.5,7.7)(-0.3,-1)
                \drawline(0,0)(4.5,0)
                \drawline(0,0)(0,7.2)
                
                \drawline(0,0)(4.5,7.2)
                
                \solid{(1,0)}\solid{(2,2)}\solid{(2,3)}\solid{(3,4)}
                \put(1.2,0.1){\makebox(0,0)[bl]{\small $a$}} 
                \put(2.2,1.9){\makebox(0,0)[ul]{\small $b$}} 
                \put(2.2,2.9){\makebox(0,0)[ul]{\small $a$}} 
                \put(3.2,3.9){\makebox(0,0)[ul]{\small $d$}} 
                
                \xdot{(1,1)}\xdot{(2,1)}\xdot{(3,1)}\xdot{(4,1)}
                
                \open{(0,0)}
                \open{(2,0)}
                \open{(3,0)}\open{(3,2)}\open{(3,3)}
                \open{(4,0)}\open{(4,2)}\open{(4,3)}\open{(4,4)}\open{(4,5)}\open{(4,6)}
                
                \put(0,-1){\makebox(4.5,0){$3/2 \leq \alpha < 5/3$}}
            \end{picture}\quad
            \begin{picture}(4.5,7.7)(-0.3,-1)
                \drawline(0,0)(4.5,0)
                \drawline(0,0)(0,7.7)
                
                \drawline(0,0)(4.5,7.7)
                
                \solid{(1,0)}\solid{(2,2)}\solid{(2,3)}\solid{(3,4)}\solid{(3,5)}
                \put(1.2,0.1){\makebox(0,0)[bl]{\small $a$}} 
                \put(2.2,1.9){\makebox(0,0)[ul]{\small $b$}}
                \put(2.2,2.9){\makebox(0,0)[ul]{\small $a$}}
                \put(3.2,3.9){\makebox(0,0)[ul]{\small $d$}}
                \put(3.2,4.9){\makebox(0,0)[ul]{\small $a$}}
                
                \xdot{(1,1)}\xdot{(2,1)}\xdot{(3,1)}\xdot{(4,1)}
                
                \open{(0,0)}
                \open{(2,0)}
                \open{(3,0)}\open{(3,2)}\open{(3,3)}
                \open{(4,0)}\open{(4,2)}\open{(4,3)}\open{(4,4)}\open{(4,5)}\open{(4,6)}
                
                \put(0,-1){\makebox(4.5,0){$5/3 \leq \alpha < 2$}}
            \end{picture}\\[5ex]
            \begin{picture}(4.5,9)(-0.3,-1)
                \drawline(0,0)(4.5,0)
                \drawline(0,0)(0,9)
                
                \drawline(0,0)(3.7,8.5)
                
                \solid{(1,0)}\solid{(1,2)}\solid{(2,3)}
                \put(1.2,0.1){\makebox(0,0)[bl]{\small $a$}} 
                \put(1.2,1.9){\makebox(0,0)[ul]{\small $a$}}
                \put(2.2,2.9){\makebox(0,0)[ul]{\small $c$}}
                
                \xdot{(1,1)}\xdot{(2,1)}\xdot{(3,1)}\xdot{(4,1)}
                
                \open{(0,0)}
                \open{(2,0)}\open{(2,2)}\open{(2,4)}
                \open{(3,0)}\open{(3,2)}\open{(3,3)}\open{(3,4)}\open{(3,5)}\open{(3,6)}
                \open{(4,0)}\open{(4,2)}\open{(4,3)}\open{(4,4)}\open{(4,5)}\open{(4,6)}\open{(4,7)}\open{(4,8)}
                
                \put(0,-1){\makebox(4.5,0){$2 \leq \alpha < 5/2$}}
            \end{picture}\quad
            \begin{picture}(4.5,9)(-0.3,-1)
                \drawline(0,0)(4.5,0)
                \drawline(0,0)(0,8.5)
                
                \drawline(0,0)(3.1,8.5)
                
                \solid{(1,0)}\solid{(1,2)}\solid{(2,3)}\solid{(2,5)}
                \put(1.2,0.1){\makebox(0,0)[bl]{\small $a$}} 
                \put(1.2,1.9){\makebox(0,0)[ul]{\small $a$}}
                \put(2.2,2.9){\makebox(0,0)[ul]{\small $c$}}
                \put(2.2,4.9){\makebox(0,0)[ul]{\small $a$}} 
                
                \xdot{(1,1)}\xdot{(2,1)}\xdot{(3,1)}\xdot{(4,1)}
                
                \open{(0,0)}
                \open{(2,0)}\open{(2,2)}\open{(2,4)}
                \open{(3,0)}\open{(3,2)}\open{(3,3)}\open{(3,4)}\open{(3,5)}\open{(3,6)}\open{(3,7)}\open{(3,8)}
                \open{(4,0)}\open{(4,2)}\open{(4,3)}\open{(4,4)}\open{(4,5)}\open{(4,6)}\open{(4,7)}\open{(4,8)}
                
                \put(0,-1){\makebox(4.5,0){$5/2 \leq \alpha < 3$}}
            \end{picture}\quad
            \begin{picture}(4.5,9)(-0.3,-1)
                \drawline(0,0)(4.5,0)
                \drawline(0,0)(0,8.5)
                
                \drawline(0,0)(1.8,8.5)
                
                \solid{(1,0)}\solid{(1,2)}\solid{(1,3)}\solid{(1,4)}
                \put(1.2,0.1){\makebox(0,0)[bl]{\small $a$}} 
                \put(1.2,1.9){\makebox(0,0)[ul]{\small $a$}}
                \put(1.2,3.1){\makebox(0,0)[ul]{\small $\vdots$}}
                \put(1.2,3.9){\makebox(0,0)[ul]{\small $a$}} 
                
                \xdot{(1,1)}\xdot{(2,1)}\xdot{(3,1)}\xdot{(4,1)}
                
                \open{(0,0)}
                \open{(2,0)}\open{(2,2)}\open{(2,3)}\open{(2,4)}\open{(2,5)}\open{(2,6)}\open{(2,7)}\open{(2,8)}
                \open{(3,0)}\open{(3,2)}\open{(3,3)}\open{(3,4)}\open{(3,5)}\open{(3,6)}\open{(3,7)}\open{(3,8)}
                \open{(4,0)}\open{(4,2)}\open{(4,3)}\open{(4,4)}\open{(4,5)}\open{(4,6)}\open{(4,7)}\open{(4,8)}
                
                \put(0,-1){\makebox(4.5,0){$\alpha \geq 3$}}
            \end{picture}
            \caption{Cases covered by Lemma \ref{lem:unique}, where a generator of $M$ has a point of $M$ directly above it. The bullets indicate generators, annotated with their type (items \emph{a}--\emph{d} of Theorem \ref{thm:1point_gens}).}
            \label{fig:unique}
        \end{figure}
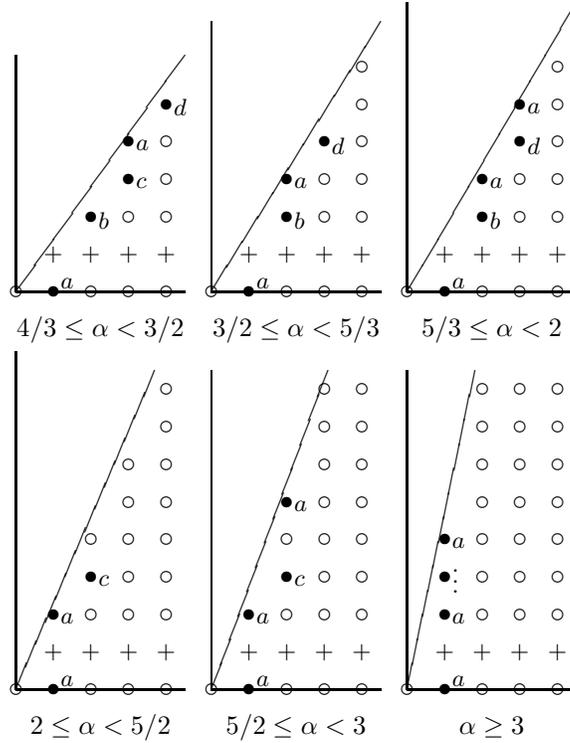

        In type \ref{it:best}, $c/d$ is a best lower approximation to $\alpha$. Then $c'/d$ must also be a best lower approximation to $\alpha$, so $v'$ is also irreducible (note that this can only occur if $d = 1$ and $\alpha \geq 3$).
        
        In type \ref{it:ord2}, for there to be even one point $v' = (\ceil{2/\alpha}, 3)$ in $M$ above the given point $v = (\ceil{2/\alpha}, 2)$, we must have the inequality
        \begin{equation}\label{eq:bd2}
            \frac{3}{\alpha} \leq \Ceil{\frac{2}{\alpha}}.
        \end{equation}
        As the ceiling augments its argument by less than $1$, we must have $\alpha > 1$, so $v$ is either $(1, 2)$ or $(2, 2)$. The first case can be excluded as $v$ is of type \ref{it:best} rather than \ref{it:ord2}. There remains the possibility that $v = (2,2)$ and $v' = (2,3)$, which appears for $3/2 \leq \alpha < 2$ and is also irreducible (of type \ref{it:best}). (Points with $c' \geq 4$ cannot occur here, as then we would have had $\alpha \geq 2$ and $(1,2) \in M$.)
        
        In type \ref{it:ord3}, we analogously find that $\alpha > 1$ and $v$ is either $(2,3)$ or $(3,3)$. Then:
        \begin{itemize}
            \item If $v = (2,3)$, we must have $2 \leq \alpha < 3$, and $v'$ is either $(2,4)$ or $(2,5)$. The vector $v' = (2,4) = 2(1,2)$ has a unique atomic decomposition. If $\alpha \geq 5/2$, then $v' = (2,5)$ is also admissible and irreducible (of type \ref{it:best}).
            \item If $v = (3,3)$, we must have $4/3 \leq \alpha < 3/2$ to get the unique possible $v' = (3,4) \in M$; this $v'$ is irreducible (of type \ref{it:best}).
        \end{itemize}
        
        For type \ref{it:c1+c2}, we first note that, since $c_2/d_2$ is the best lower approximation following $c_1/d_1 = 1/d_1$, we have
        \begin{equation}\label{eq:a_up}
            \frac{c_2}{d_2} = \frac{c_2}{c_2 d_1 - 1} \leq \alpha < \frac{c_2 - 1}{(c_2 - 1) d_1 - 1}.
        \end{equation}
        For a point $v' = (d_1 + d_2, c_1 + c_2 + 1)$ to appear above $v$ in $M$, we must have
        \begin{equation}\label{eq:a_lo}
            \alpha \geq \frac{c_1 + c_2 + 1}{d_1 + d_2} = \frac{c_2 + 2}{d_1(c_2 + 1) - 1}.
        \end{equation}
        Combining \eqref{eq:a_up} and \eqref{eq:a_lo} yields
        \begin{gather*}
            \frac{c_2 + 2}{d_1(c_2 + 1) - 1} < \frac{c_2 - 1}{(c_2 - 1) d_1 - 1}
            \intertext{which simplifies to}
            (c_2 - 1) d_1 < 3.
        \end{gather*}
        Accordingly, $c_2$ and $d_1$ must have their minimum possible values $c_2 = 2$ and $d_1 = 1$. We have $v = (3,4)$, $5/3 \leq \alpha < 2$, and $v' = (3,5)$, which is irreducible of type \ref{it:best}, completing the proof.
    \end{proof}
    
    \subsection{Relations}
    We turn our attention to understanding a set of relations for the section ring $S_D$. We again begin by looking at the genus $0$ case. In Theorem \ref{thm:1point_P1}, we have $\binom{r}{2}$ relations among the $r + 1$ generators, each led by a different quadratic monomial. These form a minimal basis for the relation ideal, as well as a Gr\"obner basis with respect to several of the commonly used term orders, including the \textbf{grevlex} order chosen by default in programs such as Sage and also used in the literature, such as in \cite{VZB}. Having a Gr\"obner basis is desirable for computations, especially if the Gr\"obner basis is also minimal.
    
    In the genus $1$ case, as one might expect, things are a bit more involved, and it is good to choose the term order judiciously so that the Gr\"obner basis is as nearly minimal as possible. The term order we use is as follows:
    
    \begin{definition} \label{def:term_order_1pt}
        Let $\{v_i\}_{i \in I}$ be the generators of $M$ as computed in Theorem \ref{thm:1point_gens}, and let $\{f_i\}_{i \in I}$ be the corresponding generators of $S_D$. Order the index set $I$ in increasing degree $d$ and, within each degree, ordered in increasing pole order $c$; the order of the generators is therefore 
        \begin{align*}
            & f_0 \prec f_\ib \prec f_\ic, & \{-1/\alpha\} &= 0 \\
            & f_0 \prec f_\ib \prec f_\ic \prec f_2 \prec f_\id \prec f_3 \prec \cdots, &
            \{-1/\alpha\} &\in (0,1/3) \\
            & f_0 \prec f_\ib \prec f_2 \prec f_\id \prec f_3 \prec \cdots, & 
            \{-1/\alpha\} &\in [1/3,1/2) \\
            & f_0 \prec f_2 \prec f_\ic \prec f_3 \prec \cdots, &
            \{-1/\alpha\} &\in [1/2,2/3) \\
            & f_0 \prec f_2 \prec f_3 \prec \cdots, & \{-1/\alpha\} &\in [2/3,1).
        \end{align*}
        Given two distinct monomials
        \[
        m^{(1)} = \prod_i v_i^{a_i^{(1)}} \textand m^{(2)} = v_i^{a_i^{(2)}},
        \]
        where $i$ ranges over the indices of all the generators in Theorem \ref{thm:1point_gens}, we declare that either $m^{(1)}$ is lower than $m^{(2)}$, written $m^{(1)} \prec m^{(2)}$, or the reverse $m^{(1)} \succ m^{(2)}$ as follows:
        \begin{enumerate}
            \item First we compare degrees: if
            \[
            \sum_i a_i^{(1)} d_i < \sum_i a_i^{(2)} d_i,
            \]
            then $m^{(1)} \prec m^{(2)}$.
            \item Then we compare pole orders: if the degrees are equal but
            \[
            \sum_i a_i^{(1)} c_i < \sum_i a_i^{(2)} c_i,
            \]
            then $m^{(1)} \prec m^{(2)}$.
            \item If the degrees and pole orders are equal, we compare exponents of the generators, starting from the highest: if $a_i^{(1)} < a_i^{(2)}$ but $a_j^{(1)} = a_j^{(2)}$ for $f_j \succ f_i$, then $m^{(1)} \prec m^{(2)}$.
        \end{enumerate}
    \end{definition}
    
    We can now state our main theorem. As in \cite{VZB}, we cannot list the relations in full detail, but at least we can provide the leading terms.
    \begin{theorem} \label{thm:1point_rels}
        Let $D = \alpha(\infty)$ be a $1$-point divisor. Denote the generators of type \ref{it:best} in Theorem \ref{thm:1point_gens} by
        \[
        f_i = u^{d_i} t_{c_i}, \quad i = 0, 2, 3, \ldots, r,
        \]
        and denote the exceptional generators of type \ref{it:ord2}, \ref{it:ord3}, and \ref{it:c1+c2} by $f_\ib$, $f_\ic$, and $f_\id$, respectively. Then a Gr\"obner basis of the relations of $S_D$ has the following leading terms:
        \begin{enumerate}
            \item All products $f_i f_j$, where $3 \leq i \leq r, 0 \leq j \leq i - 2, j \neq 1$, except possibly $f_3 f_0$ (see below);
            \item All products $f_i f_\ib$, $f_i f_\ic$, and $f_i f_\id$, where $3 \leq i \leq r$, if these exceptional generators exist;
            \item Additional relations, according to the value of $\{-1/\alpha\}$ which also controls the generators:
            \[
            \begin{tabular}{c|l|l}
                $\{-1/\alpha\} \in$ & Exc.\ gens. & Leading terms of relations \\ \hline
                $\{0\}$ & $f_\ib, f_\ic$ & $f_\ic^2$ \\
                $(0,1/3)$ & $f_\ib, f_\ic, f_\id$ & $f_\ic^2, f_\ib f_\id, f_\ic f_\id, f_\id^2, f_0f_\id, f_0f_2$ \\
                $[1/3, 1/2)$ & $f_\ib, f_\id$ & $\boxed{f_0^2f_2^2},
                f_0f_\id, f_\ib f_\id, f_\id^2$ \\
                $[1/2, 2/3)$ & $f_\ic$ & $f_\ic^2$ \\
                $[2/3, 1)$ & --- & $\boxed{f_0f_3^2}$,
                omit $f_0f_3$
            \end{tabular}
            \]
        \end{enumerate}
        Moreover, the relations comprising the Gr\"obner basis are all minimal, with the possible exception of the two cases with a non-quadratic leading term (boxed):
        \begin{itemize}
            \item The relation with the quartic leading term $f_0^2 f_2^2$ is never minimal.
            \item The relation with the cubic leading term $f_0 f_3^2$ is minimal if and only if $\{-1/\alpha\}$ belongs to the subinterval $[2/3, 3/4)$.
        \end{itemize}
        
    \end{theorem}
    
    The relations for $S_D$ are closely connected with those of the associated monoid $M$. Recall from the previous subsection that $M$ has a minimal generating set $\{v_i\}_{i \in I}$. Let $F$ be the free commutative monoid on $\size{I}$ generators $\{\tilde v_i\}_{i \in I}$ (isomorphic to $\bbZ_{\geq 0}^{\size{I}}$), and let $\pi$ be the projection map
    \begin{align*}
        \pi : F &\to M \\
        \tilde v_i &\mapsto v_i.
    \end{align*}
    To complete a presentation of $M$ is to find a (preferably finite) list of relations
    \[
    R_j \colon e_j \sim e_j'
    \]
    such that $M = F/\{R_j\}_j$, that is, such that the relation generated by the $R_j$ is precisely
    \[
    e \sim e' \iff \pi(e) = \pi(e').
    \]
    To do this systematically, we make the following definition:
    \begin{definition}
        If $v \in M$, the \textbf{minimal decomposition} of $v$ is the sum
        \[
        \MD(v) = \sum_i a_i \tilde v_i \in F
        \]
        such that $\pi(\MD(v)) = v$ and $\MD(v)$ is minimal with respect to the order from Definition \ref{def:term_order_1pt} on elements of $F$.
    \end{definition}
    \begin{remark}
        It is evident from the $d$-grading on $M$ that only finitely many such decompositions exist. Also, since $v$ is fixed, the first two steps of Definition \ref{def:term_order_1pt} may be skipped when comparing decompositions.
    \end{remark}
    
    \begin{definition}
        \label{def:Grobner monoid}
        The \textbf{Gr\"obner basis} of $M$ consists of the relations
        \[
        v \sim \MD(\pi(v))
        \]
        for all vectors $v$ satisfying the following two conditions:
        \begin{enumerate}
            \item\label{it:vv} $v \neq \MD(\pi(v))$;
            \item\label{it:ww} If $w <_F v$ (that is, $v = w + z$ with $0 \neq z \in F$), then $w = \MD(\pi(w))$.
        \end{enumerate}
        Such a vector $v \in F$ is called a \textbf{relation leader} for $M$.
    \end{definition}
    
    \begin{remark}
        It is easy to see that the Gr\"obner basis forms a presentation of $M$ as a quotient of $F$. Indeed, the relation leaders give the leading terms of a Gr\"obner basis of the monoid algebra $\Bbbk[M]$ as a quotient of the free algebra $\Bbbk[F]$ (adapting Definition \ref{def:term_order_1pt} appropriately to define a term order on $\Bbbk[F]$).
    \end{remark}
    
    The proof of Theorem \ref{thm:1point_rels} rests on the following combinatorial lemma:
    \begin{lemma} \label{lem:1point_rels_M}
        The relation leaders for $M$ are exactly the vectors $v = \sum_{i \in I} a_i \tilde v_i$ corresponding to each of the monomials $\prod_{i \in I} f_i$ claimed to be a leading term of a relation for $S_D$ in Theorem \ref{thm:2point_gens}.
    \end{lemma}
    
    Given this lemma, the proof of the theorem is not so hard:
    \begin{proof}[Proof of Theorem \ref{thm:1point_rels}]
        Given a relation leader $v \in F$, let 
        \[
        v = \sum_i a_i \tilde v_i \textand
        \MD(\pi(v)) = \sum_i b_i \tilde v_i.
        \]
        We get that $\prod_i f_i^{a_i}$ and $\prod_i f_i^{b_i}$ have the same degree and pole order, so their difference (after suitably scaling) has a pole of lower order and can be written in terms of the other generators to get a relation $r_v$ in the relation ideal of $S_D$. Since we compare monomials with reference to their pole orders, the leading term of $r_v$ is $\prod_i f_i^{a_i}$.
        
        Now, given any element $f \in S_D$ expressed as a polynomial in the generators $f_i$, we can apply monomial multiples of the relations $r_v$ to remove any relation leaders from the leading term, until either the entire sum vanishes (verifying that $f = 0$ in $S_D$) or the leading term is a minimal decomposition $\MD((d,c))$, verifying that $f$ is a nonzero element of leading degree $d$ with a pole of order $c$. This shows that the $r_v$ form a Gr\"obner basis.
        
        It remains to determine which of the relations in the Gr\"obner basis are minimal. Note that if $v = \tilde v_i + \tilde v_j$ has Hamming weight $2$, then $r_v$, which has a quadratic leading term $f_i f_j$, must be minimal because there are no relations with a term of $f_i$ or $f_j$ alone to generate this relation (or else $f_i$, respectively $f_j$, would not be a generator).
        
        This leaves the two boxed relations. The relation with quartic leading term $f_0^2 f_2^2$, in case $\{-1/\alpha\} \in [1/3,1/2)$, has the form
        \[
        f_0^2 f_2^2 = f_\ib^3 + \text{lower-order poles,}
        \]
        but since
        \[
        2v_0 + 2v_2 = v_0 + v_\ib + v_\id = 3v_\ib,
        \]
        this relation can be derived by subtracting those led by $f_\ib f_\id$ and $f_0 f_\id$ after multiplying by $f_0$ and $f_\ib$ respectively to cancel the $f_0f_\ib f_\id$ term. Hence this relation is never minimal.
        
        We now turn to the relation with cubic leading term $f_0 f_3^2$. The corresponding relation in $M$ is
        \begin{equation} \label{eq:cubic_M}
            v_0 + 2v_3 = 3v_2
        \end{equation}
        If $\{-1/\alpha\} \in [3/4, 1)$, then there is a vector $v_4 = 2v_3 - v_2 \in M$, and the relation \eqref{eq:cubic_M} is not minimal, even in $M$:
        \[
        v_0 + 2v_3 = v_0 + v_2 + v_4 = 3v_2.
        \]
        Translating from $M$ to $S_D$, we find correspondingly that the relation with leading term $f_0 f_3^2$ can be generated by the relations led by $f_2 f_4$ and $f_0 f_4$.
        
        On the other hand, if $\{-1/\alpha\} \in [2/3, 3/4)$, then $2v_3 - v_2 \notin M$, and there are no relations having a term of $f_3^2$ or $f_0 f_3$, hence no way to decompose the relation with leading term $f_0 f_3^2$.
    \end{proof}
    
    \begin{proof}[Proof of Lemma \ref{lem:1point_rels_M}]
        %
        Let $v = (d,c) \in M$. If $v \in \angle v_i v_{i+1}$ for $i \geq 2$, then we have a decomposition $v = a_i v_i + a_{i+1} v_{i+1}$ with $a_j \geq 0$. We claim that this is a minimal decomposition. Suppose not; let $v = \sum_j b_j v_j$ be a lesser one. We must have $b_j = 0$ for $j \leq i+1$. If $b_{i+1} < a_{i+1}$, then the difference $v - b_{i+1} v_{i+1}$ would be outside $\angle v_0 v_i$ and thus not expressible as a sum of generators preceding $v_{i+1}$, all of which are within that angle. So $b_{i+1} = a_{i+1}$. Similarly, if $b_i < a_{i+1}$, then the difference $v - b_{i+1} v_{i+1} - b_i v_i$ would be outside $\angle v_0 v_i$ and thus not expressible as a sum of generators preceding $v_{i+1}$, all of which are within that angle. So $b_i = v_i$ and hence all other $b_i$ are $0$.
        
        We now compute the minimal decompositions of vectors in $\angle v_0 v_2$. The generators needed are some subset of
        \[
        v_0, v_\ib, v_\ic, v_2, v_\id, v_3
        \]
        in the various cases. The resulting minimal decompositions are shown in Table \ref{tab:min_decomps}. Each individual decomposition is not hard to prove minimal. The pattern of the decompositions is as follows: the generating sets $\{v_0, v_\ib\}$, $\{v_\ib, 2\}$, or (if $\{-1/\alpha\} \geq 1/2$) $\{v_0, v_2\}$ generates a sublattice of index $2$ inside the appropriate angle. If the desired vector $v$ lies in this sublattice as tested by the parity of $a_1$, its minimal decomposition is an integer combination of those two generators; otherwise there is need for a single copy of $v_\ic$, $v_\id$, or $v_3$.

        \begin{table}[htbp]
            \centering
            \begin{tabular}{r|l|l}
                $\left\{-\dfrac{1}{\alpha}\right\} \in$ & Generators & Minimal decompositions
                \\
                \hline
                
                $\{0\}$ &
                \begin{tabular}[t]{@{}l}
                    $v_0$ \\
                    $v_\ib = 2v_1$ \\
                    $v_\ic = 3v_1$ \\
                \end{tabular} &
                \begin{tabular}[t]{@{}l}
                    $a_1 v_1 + a_0 v_0 = \begin{dcases*}
                        \frac{\vphantom{A^2}a_1}{2} v_\ib + a_0 v_0, & $a_1 \geq 0$ even \\
                        v_\ic + \frac{a_1 - 3}{2} v_\ib + a_0 v_0, & $a_1 \geq 3$ odd
                    \end{dcases*}$
                \end{tabular} \\
                
                $(0,1/3)$ &
                \begin{tabular}[t]{@{}l}
                    $v_0$ \\
                    $v_\ib = 2 v_1$ \\
                    $v_\ic = 3v_1$ \\
                    $v_2$ \\
                    $v_\id = v_1 + v_2$
                \end{tabular} &
                \begin{tabular}[t]{@{}l}
                    $a_1 v_1 + a_0 v_0 = \begin{dcases*}
                        \frac{a_1}{2} v_\ib + a_0 v_0, & $a_1 \geq 0$ even \\
                        v_\ic + \frac{a_1 - 3}{2} v_\ib + a_0 v_0, & $a_1 \geq 3$ odd
                    \end{dcases*}$ \\
                    $a_2 v_2 + a_1 v_1 = \begin{dcases*}
                        a_2 v_2 + \frac{a_1}{2} v_\ib, & $a_1 \geq 0$ even \\
                        v_\id + (a_2 - 1)v_2, & $a_1 = 1$ \\
                        a_2 v_2 + v_\ic + \frac{a_1 - 3}{2} v_\ib, & $a_1 \geq 3$ odd
                    \end{dcases*}$
                \end{tabular} \\
                
                $[1/3,1/2)$ &
                \begin{tabular}[t]{@{}l}
                    $v_0$ \\
                    $v_\ib = 2 v_1$ \\
                    $v_2 = 3v_1 - v_0$ \\
                    $v_\id = v_1 + v_2$
                \end{tabular} &
                \begin{tabular}[t]{@{}l}
                    $a_1 v_1 + a_0 v_0 = \begin{dcases*}
                        \frac{a_1}{2} v_\ib + a_0 v_0, & $a_1 \geq 0$ even \\
                        v_2 + \frac{a_1 - 3}{2} v_\ib + (a_0 + 1) v_0, & $a_1 \geq 3$ odd
                    \end{dcases*}$ \\
                    $a_2 v_2 + a_1 v_1 = \begin{dcases*}
                        a_2 v_2 + \frac{a_1}{2} v_\ib, & $a_1 \geq 0$ even \\
                        v_\id + (a_2 - 1)v_2, & $a_1 = 1$ \\
                        (a_2 + 1) v_2 + \frac{a_1 - 3}{2} v_\ib + v_0, & $a_1 \geq 3$ odd
                    \end{dcases*}$
                \end{tabular} \\
                
                $[1/2,2/3)$ &
                \begin{tabular}[t]{@{}l}
                    $v_0$ \\
                    $v_2 = 2v_1 - v_0$ \\
                    $v_\ic = 3v_1 - v_0$
                \end{tabular} &
                \begin{tabular}[t]{@{}l}
                    $a_1 v_1 + a_0 v_0 = \begin{dcases*}
                        \frac{a_1}{2} v_2 + \left(\frac{a_1}{2} + a_0\right) v_0, & $a_1 \geq 0$ even \\
                        v_\ic + \frac{a_1 - 3}{2} v_2 + \left(\frac{a_1 - 3}{2} + a_0 \right) v_0, & $a_1 \geq 3$ odd
                    \end{dcases*}$ \\
                    $a_2 v_2 + a_1 v_1 = \begin{dcases*}
                        \left(a_2 + \frac{a_1}{2}\right) v_2 + \frac{a_1}{2} v_0, & $a_1 \geq 0$ even \\
                        v_\ic + \left(a_2 + \frac{a_1 - 1}{2}\right)v_2 + \frac{a_1 - 1}{2} v_0, & $a_1 \geq 1$ odd
                    \end{dcases*}$
                \end{tabular} \\
                
                $[2/3,1)$ &
                \begin{tabular}[t]{@{}l}
                    $v_0$ \\
                    $v_2 = 2v_1 - v_0$ \\
                    $v_3 = 3v_1 - 2v_0$
                \end{tabular} &
                \begin{tabular}[t]{@{}l}
                    $a_1 v_1 + a_0 v_0 = \begin{dcases*}
                        \frac{a_1}{2} v_2 + \left(\frac{a_1}{2} + a_0\right) v_0, & $a_1 \geq 0$ even \\
                        v_3 + \frac{a_1 - 3}{2} v_2 + \left(\frac{a_1 - 3}{2} + a_0\right)v_0, & $a_1 \geq 3$ odd
                    \end{dcases*}$ \\
                    $a_2 v_2 + a_1 v_1 = \begin{dcases*}
                        \left(a_2 + \frac{a_1}{2}\right) v_2 + \frac{a_1}{2} v_0, & $a_1 \geq 0$ even \\
                        v_3 + \left(a_2 + \frac{a_1 - 3}{2}\right)v_2 + \frac{a_1 + 1}{2} v_0, & $a_1 \geq 1$ odd
                    \end{dcases*}$
                \end{tabular}
                
            \end{tabular}
            \caption{Minimal decompositions of vectors of $M$ in $\angle v_0v_2$}
            \label{tab:min_decomps}
        \end{table}
        
        Looking over the minimal decompositions that we have found, we immediately see that the sums $v_i + v_j$, where $i \geq 3$ and $0 \leq j \leq i - 2$, do not appear and therefore are relation leaders, with one exception: if $\{-1/\alpha\} \in [2/3,1)$, then $v_0 + v_3$ is the minimal decomposition of $v_1 + v_2$. The same argument applies to the sums $v_i + v_j$ where $i \geq 3$ and $j$ is one of the special symbols \ref{it:ord2}, \ref{it:ord3}, and \ref{it:c1+c2}. It remains only to check sums of $v_0$, $v_\ib$, $v_\ic$, $v_\id$, $v_2$, and (if $\{-1/\alpha\} \in [2/3,1)$) $v_3$, and in each case we get the desired result:
        \begin{itemize}
            \item For $\{-1/\alpha\} = 0$, as the only generators are $f_0$, $f_\ib$, and $f_\ic$, there will only be one relation leader, namely $2v_\ic$. Any decomposition with at most one copy of $v_\ic$ is minimal by our computations.
            \item For $\{-1/\alpha\} \in (0,1/3)$, any decomposition not containing the relation leaders $2v_\ic$, $v_\ib + v_\id$, $v_\ic + v_\id$, $2v_\id$, $v_0 + v_\id$, or $v_0 + v_2$ is of one of the forms
            \begin{gather*}
                a_0 v_0 + a_\ib v_\ib, \quad a_0 v_0 + a_\ib v_\ib + v_\ic, \\
                a_\ib v_\ib + a_2 v_2, \quad a_\ib v_\ib + a_2 v_2 + v_\ic, \quad a_2 v_2 + v_\id
            \end{gather*}
            and hence is minimal.
            \item For $\{-1/\alpha\} \in [1/3,1/2)$, any decomposition not containing the relation leaders $2v_\ic$, $v_\ib + v_\id$, $v_\ic + v_\id$, $2v_\id$, $v_0 + v_\id$, or  $2v_0 + 2v_2$ is of one of the forms
            \begin{gather*}
                a_0 v_0 + a_\ib v_\ib, \quad a_0 v_0 + a_\ib v_\ib + v_2, \\
                a_\ib v_\ib + a_2 v_2, \quad a_\ib v_\ib + a_2 v_2 + v_0, \quad a_2 v_2 + v_\id
            \end{gather*}
            and hence is minimal.
            \item For $\{-1/\alpha\} \in [1/2,2/3)$, any decomposition not containing the relation leader $2v_\ic$ is of one of the forms
            \begin{gather*}
                a_0 v_0 + a_2 v_2, \quad a_0 v_0 + a_2 v_2 + v_\ic
            \end{gather*}
            and hence is minimal.
            \item For $\{-1/\alpha\} \in [2/3,1)$, any decomposition not containing the relation leader $v_0 + 2v_3$ is of one of the forms
            \begin{gather*}
                a_0 v_0 + a_2 v_2, \quad a_0 v_0 + a_2 v_2 + v_3, \quad a_2 v_2 + a_3 v_3
            \end{gather*}
            and hence is minimal.
        \end{itemize}
        This completes the proof of the lemma.
    \end{proof}
    
    \section{The effective two-point case}
    \label{section: 2pt effective}
    
    Let $D = \alpha^{(1)}(P^{(1)}) + \alpha^{(2)}(P^{(2)})$ be an effective $\bbQ$-divisor on an elliptic curve $C$ supported on two points $P^{(k)}$. In this section, we study the structure of the associated section ring $S_D$. We may assume that $\alpha^{(1)} \geq \alpha^{(2)}$ since the roles of the points $P^{(k)}$ may be interchanged.
    
    \subsection{Generators}
    
    For $c \in \bbZ_{\geq 0}$, $c \neq 1$, and for $i \in \{1,2\}$, denote by $t_c^{(k)}$ a function on $C$ whose polar divisor is $c(P^{(i)})$. Also, let $w$ be the function on $C$ whose polar divisor is $P^{(1)} + P^{(2)}$. (Such a function is unique up to scaling and adding constants.) Note the following:
    
    \begin{lemma} \label{lem:H0_2pt}
        Let $D = a^{(1)}(P^{(1)}) + a^{(2)}(P^{(2)})$ be a nonzero effective $\bbZ$-divisor supported at two points. The linear system of functions $H^0(D)$ has dimension $a^{(1)} + a^{(2)}$, with a basis as follows:
        \begin{itemize}
            \item $\{1, t_2^{(1)}, \ldots, t_{a^{(1)}}^{(1)}\}$ if $a^{(1)} > a^{(2)} = 0$;
            \item $\{1, t_2^{(2)}, \ldots, t_{a^{(2)}}^{(2)}\}$ if $a^{(2)} > a^{(1)} = 0$;
            \item $\{1, w, t_2^{(1)}, \ldots, t_{a^{(1)}}^{(1)}, t_2^{(2)}, \ldots, t_{a^{(2)}}^{(2)}\}$ if $a^{(1)}$ and $a^{(2)}$ are positive.
        \end{itemize}
    \end{lemma}
    \begin{proof}
        The dimension of $H^0(D)$ is given by the Riemann--Roch theorem. We check that the claimed functions are linearly independent because they have different orders at $\infty$, so they must form a basis.
    \end{proof}
    
    We now state the generators of the section ring. Our description generalizes the one-point case (Theorem \ref{thm:1point_gens}):
    \begin{theorem}\label{thm:2point_gens}   
        Let $D = \alpha^{(1)}(P^{(1)}) + \alpha^{(2)}(P^{(2)})$ be an effective $\bbQ$-divisor on an elliptic curve $C$ supported on two points $P^{(k)}$, with $\alpha^{(1)} \geq \alpha^{(2)}$. Let 
        \[
        0 = \frac{c_0^{(k)}}{d_0^{(k)}} < \frac{c_1^{(k)}}{d_1^{(k)}} < \cdots < \frac{c_{r^{(k)}}^{(k)}}{d_{r^{(k)}}^{(k)}} = \alpha^{(k)}
        \]
        be the best lower approximations to $\alpha^{(k)}$. Notice that the denominators
        \[
        d_1^{(k)} = \Ceil{1/\alpha^{(k)}}
        \]
        satisfy $d_1^{(1)} \leq d_1^{(2)}$.
        
        Then $S_D$ has a minimal system of generators of the following forms:
        \begin{enumerate}[$($a$)$]
            \item $f_i^{(k)} = t_{c_i^{(k)}}^{(k)} u^{d_i^{(k)}}$ for $k \in \{1, 2\}$ and $i = 0, 2, 3, \ldots, r^{(k)}$ (including $f_0 = u$ for $i = 0$)
            \item $f_\ib = t_{2}^{(1)} u^{\ceil{2/\alpha^{(1)}}}$ if $\{-1/\alpha^{(1)}\} \in [0,1/2)$;
            \item $f_\ic = t_{3}^{(1)} u^{\ceil{3/\alpha^{(1)}}}$ if $\{-1/\alpha^{(1)}\} \in [0,1/3) \union [1/2, 2/3)$ and $\ceil{1/\alpha^{(2)}} > \ceil{1/\alpha^{(1)}}$;
            \item $f_\id = t_{c_1^{(1)} + c_2^{(2)}}^{(1)} u^{d_1^{(1)} + d_2^{(1)}}$, if $\{-1/\alpha^{(1)}\} \in (0,1/2)$ and $\ceil{1/\alpha^{(2)}} > \ceil{1/\alpha^{(1)}}$;
            \item $f_w = w u^{d_1^{(2)}}$.
        \end{enumerate}
    \end{theorem}
    
    \begin{proof}
        Lemma \ref{lem:H0_2pt} suggests the following strategy. Write $D = D^{(1)} + D^{(2)}$ where
        \[
        D^{(k)} = \alpha^{(k)} P^{(k)}.
        \]
        Then, as a $\Bbbk$-vector space,
        \begin{equation} \label{eq:S_D decompose}
            S_D = S_{D^{(1)}} + S_{D^{(2)}} + \Bbbk[u] u^{d_1^{(2)}} w,
        \end{equation}
        because $w$ is the only basis element for any graded piece $u^n H^0(nD)$ of $S_D$ that does not already appear in either $S_{D^{(1)}}$ or $S_{D^{(2)}}$, and it first appears in degree $\ceil{1/\alpha^{(2)}} = d_1^{(2)}$. Hence we can get a generating set for $S_D$ by aggregating generating sets for $S_{D^{(1)}}$ and $S_{D^{(2)}}$ and adjoining the single added generator $u^{d_1^{(2)}} w$. We must then pare this set down to a minimal generating set. From the general theory of graded algebras, the \emph{degrees} of the minimal generators are uniquely determined; but the generators themselves and their pole orders at the $P^{(k)}$) are not. Observe that the generator $w u^{d_1^{(2)}}$ is minimal, as it is the lowest-degree element in $S_D$ that does not lie in the subring $S_{D^{(1)}} S_{D^{(2)}}$ generated by functions with poles at only one of the two given points $P^{(k)}$.
        
        A generator of type \ref{it:best} in each $S_{D^{(k)}}$ (the labeling coming from Theorem \ref{thm:1point_gens}) will always remain minimal in $S_D$, as there is no way to get a pole of order $c_j^{(k)}$ by combining elements of lower degrees, by definition of best lower approximation. We consider the other types in turn.
        
        We first claim that the generators of types \ref{it:ord2}, \ref{it:ord3}, and \ref{it:c1+c2} in $S_{D^{(2)}}$, if any, are \emph{never} minimal in $S_D$ and can be removed. By the proof of Theorem \ref{thm:1point_gens}, all of these correspond to vectors $(d,c)$ that are minimal generators of the monoid
        \[
        M^{(2)} = \{(d,c) \in \bbZ^2 : 0 \leq c \leq \alpha^{(2)} d, c \neq 1\}
        \]
        but not of the simpler monoid
        \[
        M_0^{(2)} = \{(d,c) \in \bbZ^2 : 0 \leq c \leq \alpha^{(2)} d\}.
        \]
        But $S_D$, unlike $S_{D^{(2)}}$, contains homogeneous elements $f$ achieving every vector
        \[
        (d,c) = (\deg f, -{\ord}_{P^{(2)}} f)
        \]
        in $M_0^{(2)}$: take $w u^d$ if $c = 1$, and otherwise take the appropriate element of $S_{D^{(2)}}$. Consequently, given a generator $g \in S_D$ whose associated vector $(d,c) = (\deg g, -{\ord}_{P^{(2)}} g)$ is reducible in $M_0^{(2)}$, we can multiply elements of $S_D$ of lower degrees to achieve the pole order $c$ and subtract this from $g$ to leave a pole of lower order. This lower order can again be achieved by a product of generators besides $g$ (note that generators of types \ref{it:ord2}, \ref{it:ord3}, and \ref{it:c1+c2} always appear in distinct degrees, referring to Lemma \ref{lem:unique} and Figure \ref{fig:unique}), and continuing this way, we arrive at a function with no pole at $P^{(2)}$, that is, an element of $S_{D^{(1)}}$. Hence $g$ is a polynomial in the other generators of $S_D$.
        
        We now investigate under what conditions the generators of types \ref{it:ord2}, \ref{it:ord3}, and \ref{it:c1+c2} in $S_{D^{(1)}}$ remain minimal in $S_D$. For \ref{it:ord2}, the generator $f_\ib = t_2^{(1)} u^{\ceil{2/\alpha^{(1)}}}$ of degree $d = \ceil{2/\alpha^{(1)}}$ is the first appearance in $S_D$ of a function with a double pole at $P^{(1)}$. By assumption, there is no other generator of $S_{D^{(1)}}$ with a double pole in this degree, so the only way to eliminate it is to multiply two functions of lower degree with simple poles at $P^{(1)}$. The first function of this sort is
        \begin{equation} \label{eq:w^2}
            f_w^2 = w u^{\ceil{1/\alpha^{(2)}}} \cdot w u^{\ceil{1/\alpha^{(2)}}}
        \end{equation}
        of degree
        \[
        2 \Ceil{\frac{1}{\alpha^{(2)}}} \geq 2 \Ceil{\frac{1}{\alpha^{(1)}}}
        \stackrel{*}{=} \Ceil{\frac{2}{\alpha^{(1)}}} = d,
        \]
        the starred equality holding since $\{-1/\alpha^{(1)}\} \in [0,1/2)$.
        So the only way that this generator can be non-minimal is if equality holds, and in particular, $S_{D^{(1)}}$ and $S_{D^{(2)}}$ look alike up to degree $d$. But the unique product \eqref{eq:w^2} in this degree has double poles at both $P^{(1)}$ and $P^{(2)}$, and since we already threw out the generator $t_2^{(2)} u^{d}$ of type \ref{it:ord2} in $S_{D^{(2)}}$, there is no way to cancel out the double pole at $P^{(2)}$. Hence the generator $f_\ib$, if it appears in $S_{D^{(1)}}$, always remains minimal in $S_D$, as claimed.
        
        Next, we look at type \ref{it:ord3}. Here we have a generator $f_\ic = t_3^{(1)} u^{\ceil{3/\alpha^{(1)}}}$ of degree $d = 3/\ceil{3/\alpha^{(1)}}$, which is the least degree in which there appears a triple pole at $\alpha^{(1)}$. To cancel out this pole, we must multiply two elements of lower degree with a single and a double pole at $P^{(1)}$. The first function of this sort is
        \begin{equation} \label{eq:wt_2}
            w u^{\ceil{1/\alpha^{(2)}}} \cdot t_2^{(1)} u^{\ceil{2/\alpha^{(1)}}} = \begin{cases}
                f_w f_\ib, & \{-1/\alpha^{(1)}\} \in [0,1/2) \\
                f_w f_2^{(1)}, & \{-1/\alpha^{(1)}\} \in [1/2,1)
            \end{cases}
        \end{equation}
        of degree
        \begin{equation} \label{eq:numerics c}
            \Ceil{\frac{1}{\alpha^{(2)}}} + \Ceil{\frac{2}{\alpha^{(1)}}} \geq
            \Ceil{\frac{1}{\alpha^{(1)}}} + \Ceil{\frac{2}{\alpha^{(1)}}} \stackrel{*}{=}
            \Ceil{\frac{3}{\alpha^{(1)}}} = d,
        \end{equation}
        the starred equality holding since $\{-1/\alpha^{(1)}\} \in [0,1/3) \union [1/2, 2/3)$. Again, the generator is therefore minimal unless equality holds. If equality holds, then the function \eqref{eq:wt_2} has a triple pole at $P^{(1)}$ as well as a simple pole at $P^{(2)}$ which can be canceled by adding the appropriate multiple of $wu^d$, yielding an element of $S_{D^{(1)}}$ with the desired triple pole. Accordingly, the generator $g$ of type \ref{it:ord3} is minimal if and only if equality does \emph{not} hold in \eqref{eq:numerics c}, as claimed.
        
        Finally, suppose $S_{D^{(1)}}$ has a generator of type \ref{it:c1+c2}, which is of the form $f_\id = t_c^{(1)} u^d$ where
        \[
        d = d_1^{(1)} + d_2^{(1)}, \quad c = c_1^{(1)} + c_2^{(1)} = 1 + c_2^{(1)}
        \]
        Note that this is only the second degree, after $d_2^{(1)}$, in which $S_D$ is strictly larger than the subring $S' = S_{1/d_1^{(1)} P^{(1)} + \alpha^{(2)} P^{(2)}}$ where the pole order at $P^{(1)}$ is limited by the first best lower approximation $1/d_1^{(1)}$. To eliminate $f_\id$, we must multiply two elements of lower degree at least one of which lies outside $S'$. Hence we must use
        \[
        f = t_{c_2^{(1)}} u^{d_2^{(1)}},
        \]
        a generator of $S_{D^{(1)}}$ of type \ref{it:best}. We must multiply it by an element $h$ of degree $d_1^{(1)}$ with at least a simple pole at $P^{(1)}$. But a double pole at $P^{(1)}$ does not appear until degree $\ceil{2/\alpha^{(1)}} \geq 2d_1^{(1)} - 1$, which is too high unless $d_1^{(1)} = 1$ and $\alpha^{(1)} \geq 2$, and here there cannot be a generator of type \ref{it:c1+c2}, because then the generator $t_3^{(1)} u^2$ is of type \ref{it:best} or \ref{it:ord3} rather than \ref{it:c1+c2}. So $h$ must have a simple pole at $P^{(1)}$, which only happens when the generator $h = w u^{d_1^{(2)}}$ has low enough degree, namely when
        \begin{equation}\label{eq:numerics d}
            \Ceil{\frac{1}{\alpha^{(2)}}} = \Ceil{\frac{1}{\alpha^{(1)}}}.
        \end{equation}
        We can then multiply $f \cdot h$ to get a function in degree $d$ with the desired pole order $c$ at $P^{(1)}$ and a simple pole at $P^{(2)}$, which can be canceled by adding the appropriate multiple of $wu^d$. Accordingly, the generator $f_\id$ is minimal if and only if equality does \emph{not} hold in \eqref{eq:numerics d}, as claimed.
    \end{proof}
    
    \subsection{Relations}
    In this section we state and prove the relations among the generators in the effective two-point case. It will be noted that, in Theorem \ref{thm:2point_gens}, the form of the generators is quite different depending on whether \eqref{eq:numerics d} holds or not. This bifurcation in turn affects the term order we choose and the form of the relations, and hence we divide the statement and proof into two.
    \subsubsection{The unequal ceilings case}
    In this section we assume that
    \begin{equation} \label{eq:unequal}
        \Ceil{\frac{1}{\alpha^{(2)}}} < \Ceil{\frac{1}{\alpha^{(1)}}}.
    \end{equation}
    
    \begin{definition} \label{def:term_order_unequal}
        Let $D = \alpha^{(1)} P^{(1)} + \alpha^{(2)} P^{(2)}$ be an effective $\bbQ$-divisor on $C$ supported at two points, and assume \eqref{eq:unequal}. We order the generators of $S_D$
        \begin{enumerate}[(1)]
            \item first by pole order at $P^{(2)}$,
            \item then by degree,
            \item then by pole order at $P^{(1)}$.
        \end{enumerate}
        We order the monomials in the generators of $S_D$
        \begin{enumerate}[(1)]
            \item first by pole order at $P^{(2)}$,
            \item then by degree,
            \item then by pole order at $P^{(1)}$,
            \item then by the exponents of the generators, starting with the highest generator.
        \end{enumerate}
    \end{definition}
    \begin{remark}
        By sorting first by pole order at $P^{(2)}$, we ensure that the generators of the subring $S_{D^{(1)}}$, and monomials therein, appear first in the ordering, and in the same order as in the $1$-point case (Definition \ref{def:term_order_1pt}). For instance, if $\{-1/\alpha^{(1)}\} \in (0, 1/3)$, the ordering of the generators is
        \[
        u = f_0 \prec f_\ib \prec f_\ic \prec f_2^{(1)} \prec f_\id \prec f_3^{(1)} \prec \cdots \prec f_{r^{(1)}}^{(1)} \prec f_w \prec f_2^{(2)} \prec \cdots \prec f_{r^{(2)}}^{(2)}.
        \]
    \end{remark}
    
    \begin{theorem} \label{thm:2point_rels_unequal}
        With this term order, a Gr\"obner basis for the relation ideal of $S_D$ has the following leading terms:
        \begin{enumerate}
            \item The same leading terms of the relations among the $f_i^{(1)}$, $f_\ib$, $f_\ic$, and $f_\id$ that obtain in the one-point case for $S_{D^{(1)}}$ in Theorem \ref{thm:1point_rels};
            \item $u \cdot f_i^{(2)}$, $i \geq 2$;
            \item $f_w \cdot f_i^{(2)}$, $i \geq 3$;
            \item $f_i^{(2)} \cdot f_j^{(2)}$, $i \geq j + 2$;
            \item All products $f^{(1)} f^{(2)}$ where $f^{(1)}$ is of one of the forms $f_i^{(1)}$ ($i \geq 2$), $f_\ib$, $f_\ic$, or $f_\id$ and $f^{(2)}$ is of one of the forms $f_j^{(2)}$ ($j \geq 2$) or $f_w$.
        \end{enumerate}
    \end{theorem}
    
    \begin{remark}
        For the curious reader, this is the case in which we get the maximum of nine exceptional relations claimed in Theorem \ref{thm:2point_intro}: six in $S_{D^{(1)}}$ from the case of Theorem \ref{thm:1point_intro} where $\{-1/\alpha\} \in (0,1/3)$, plus the three added ones with leading terms $f_w f_\ib$, $f_w f_\ic$, and $f_w f_\id$.
    \end{remark}
    
    \begin{proof}
        Because we ordered the generators of $S_{D^{(1)}}$ (the ``old'' generators) before all others (the ``new'' generators), the Gr\"obner basis for the relations among the old generators is unchanged from the one-point case. We call these the ``old'' relations.
        
        A $\Bbbk$-basis for the quotient space $S_{D}/S_{D^{(1)}}$ consists of one function $f_{(d,c)}^{(2)}$ of degree $d$ having pole order $c$ at $P^{(2)}$, for each $(d,c)$ with $1 \leq c \leq \alpha^{(2)} d$. Note that $c = 1$ need not be excluded now. Consequently, the relations between the new generators closely parallel the genus zero case. For $(d,c)$ in the angle $\angle v_i^{(2)} v_{i+1}^{(2)}$, the minimal monomial achieving degree $d$ and pole order $c$ is a product of the appropriate powers of the two consecutive generators $f_i^{(2)}$ and $f_{i+1}^{(2)}$, where $f_1^{(2)}$ must be replaced by $f_w$. Consequently, any product of two nonconsecutive new generators, or of a new and an old generator, is the leading term of a relation (a ``new'' relation).
        
        The new relations are all minimal because their leading terms are quadratic. It remains to consider whether an old, minimal relation can become non-minimal when the new generators and relations are added. (An old, non-minimal relation obviously remains non-minimal here.) Referring to Theorem \ref{thm:1point_rels}, there was only one case where a relation with a non-quadratic leading term was nonetheless minimal: the case $\{-1/\alpha^{(1)}\} \in [2/3, 3/4)$. Here there are generators $u = f_0, f_2^{(1)}, f_3^{(1)}$ corresponding to best lower approximations
        \[
        \frac{c_0}{d_0} = 1, \quad \frac{c_2}{d_2} = \frac{2}{2n-1}, \quad \frac{c_3}{d_3} = \frac{3}{3n - 2}
        \]
        (letting $n = d_1 = \ceil{1/\alpha^{(1)}}$). The relation in question has leading term $u {f_3^{(1)}}^2$, so if it is not minimal, then some other relation must have a term of $u f_3^{(1)}$ or ${f_3^{(1)}}^2$ to cancel it out. We claim there is no relation having either of these terms. This is clear for ${f_3^{(1)}}^2$ because its pole order of $6$ at $P^{(1)}$ is the largest possible in degree $6n - 4$, except possibly $\ord_{P^{(1)}}(f_4^{(1)}) = -7$ when $n = 1$, and no other monomial achieves this pole order. As to $u f_3$, we observe that the only possible monomials in degree $3n - 1$ having a pole of order at least $3$ at $P^{(1)}$ are
        \begin{equation} \label{eq:34567}
            u f_3, \quad {f_2^{(1)}}^2, \quad f_2^{(1)} f_3^{(1)}, \quad {f_3^{(1)}}^2, \quad f_4^{(1)}.
        \end{equation}
        (since $f_2^{(1)} f_w$, the first monomial of this sort outside $S_{D^{(1)}}$, has degree at least $(2n - 1) + (n + 1) > 3n - 1$). 
        The functions \eqref{eq:34567} have poles of distinct orders $3, 4, 5, 6, 7$ at $P^{(1)}$ and thus cannot figure in any relation, completing the proof.
    \end{proof}
    
    \subsubsection{The equal ceilings case}
    In this section we assume that
    \begin{equation} \label{eq:equal}
        \Ceil{\frac{1}{\alpha^{(1)}}} = \Ceil{\frac{1}{\alpha^{(2)}}}.
    \end{equation}
    
    \begin{definition} \label{def:term_order_equal}
        Let $D = \alpha^{(1)} P^{(1)} + \alpha^{(2)} P^{(2)}$ be an effective $\bbQ$-divisor on $C$ supported at two points, and assume \eqref{eq:equal}. We order the generators of $S_D$ in the following way:
        \[
        u = f_0 \prec f_w \prec \big[f_\ib\big] \prec
        f_2^{(1)} \prec \cdots \prec f_{r^{(1)}}^{(1)} \prec
        f_2^{(2)} \prec \cdots \prec f_{r^{(2)}}^{(2)},
        \]
        with the brackets because $f_\ib$ appears only when $\{-1/\alpha^{(1)}\} \in [0,1/2)$.
        We order monomials in the generators by the exponents of the generators, largest first (i.e.\ lex order).
    \end{definition}
    \begin{remark}
        Note the absence of comparison of degrees or pole orders in this term order, in contrast to Definition \ref{def:term_order_unequal}. It would be desirable to use a more uniform term order for all effective two-point divisors, but this leads to difficulties such as a profusion of cases and a heightened number of non-minimal relations in the Gr\"obner basis.
    \end{remark}
    
    \begin{theorem} \label{thm:2point_rels_equal}
        With this term order, a Gr\"obner basis for the relation ideal of $S_D$ has the following leading terms:
        \begin{enumerate}[$($a$)$]
            \item If $\{-1/\alpha^{(1)}\} \in [0,1/2)$:
            \begin{enumerate}[1.]
                \item $f_\ib^2$;
                \item $u f_i^{(k)}$, $i \geq 2$, $k \in \{1,2\}$;
                \item $f_w f_i^{(k)}$, $i \geq 3$, $k \in \{1,2\}$;
                \item $f_\ib f_i^{(k)}$, $i \geq 2$, $k \in \{1,2\}$;
                \item $f_i^{(k)} f_j^{(k)}$, $i \geq j + 2$, $k \in \{1,2\}$;
                \item $f_i^{(1)} f_j^{(2)}$, $i \geq 2$, $j \geq 2$.
            \end{enumerate}
            \item If $\{-1/\alpha^{(1)}\} \in [1/2,1)$:
            \begin{enumerate}[1.]
                \item $\boxed{f_2^{(1)} f_w^2}$;
                \item $u f_i^{(1)}$, $i \geq 3$;
                \item $u f_i^{(2)}$, $i \geq 2$;
                \item $f_w f_i^{(k)}$, $i \geq 3$, $k \in \{1,2\}$;
                \item $f_i^{(k)} f_j^{(k)}$, $i \geq j + 2$, $k \in \{1,2\}$;
                \item $f_i^{(1)} f_j^{(2)}$, $i \geq 2$, $j \geq 2$.
            \end{enumerate}
        \end{enumerate}
        Moreover, the relations comprising the Gr\"obner basis are all minimal, with the possible exception of the one with a cubic leading term (boxed), which is minimal if and only if
        \[
        \{-1/\alpha^{(1)}\} \in [{1}/{2}, {2}/{3}) \textand
        \{-1/\alpha^{(2)}\} \in [0, {1}/{2}).
        \]
    \end{theorem}
    
    \begin{proof}
        The proof of this theorem is somewhat different from the preceding ones owing to the different term order. Let $n = \ceil{1/\alpha^{(1)}} = \ceil{1/\alpha^{(2)}}$, and define the subdivisors
        \begin{align*}
            D' &= \frac{2}{\ceil{2/\alpha^{(1)}}} P^{(1)} + \frac{1}{n} P^{(2)} = \begin{cases}
                \frac{1}{n} P^{(1)} + \frac{1}{n} P^{(2)}, \{-1/\alpha^{(1)}\} \in [0,1/2) \\
                \frac{2}{2n - 1} P^{(1)} + \frac{1}{n} P^{(2)}, \{-1/\alpha^{(1)}\} \in [1/2,1)
            \end{cases} \\
            D'' &= \frac{2}{\ceil{2/\alpha^{(1)}}} P^{(1)} + \alpha^{(2)} P^{(2)}.
        \end{align*}
        The significance of the resulting filtration of the section rings $S_{D'} \subseteq S_{D''} \subseteq S_D$ is that
        \begin{enumerate}[$($a$)$]
            \item $S_{D'}$ is generated by the lowest three generators: $u$, $f_w$, and either $f_\ib$ ($\{-1/\alpha^{(1)}\} \in [0,1/2)$) or $f_2^{(1)}$ ($\{-1/\alpha^{(1)}\} \in [1/2,1)$);
            \item $S_{D''}$ is generated by $S_{D'}$ and the remaining generators $f_i^{(1)}$;
            \item $S_D$ is generated by $S_{D''}$ and the remaining generators $f_i^{(2)}$.
        \end{enumerate}
        These claims are not hard to show. For brevity, we focus on the case $\{-1/\alpha^{(1)}\} \in [1/2,1)$, the other being analogous. Here $S_D$ has a relation because the element $f_2^{(1)} f_w^2$, with degree and pole orders $(d, c^{(1)}, c^{(2)}) = (4n - 1, 4, 2)$, can be expressed as a linear combination of terms ${f_2^{(1)}}^2 u$ $(4n - 1, 4, 0)$, $f_w^2 u^{2n-1}$ $(4n-1, 2, 2)$, and elements with lower pole orders (namely $f_2^{(1)}f_w u^n, f_2^{(1)} u^{2n}, f_w u^{3n-1}$, and $u^{4n-1}$). Applying this relation, we can express any polynomial in the first three generators as a linear combination of terms $f_w^i u^j$, ${f_2^{(1)}}^i u^j$, and ${f_2^{(1)}}^i f_w u^j$. Specifically, in degree $d \geq n$, we have the elements
        \begin{equation} \label{eq:basis_S_D'}
            \begin{aligned}
                &f_w^i u^{d - ni}                 ,& 0 &\leq i \leq \frac{d}{n} \\
                &{f_2^{(1)}}^i u^{d - (2n-1)i}    ,& 1 &\leq i \leq \frac{d}{2n-1} \\
                &{f_2^{(1)}}^i f_w u^{d - (2n-1)i},& 1 &\leq i \leq \frac{d - n}{2n - 1}.
            \end{aligned}
        \end{equation}
        Comparing pole orders shows these are all linearly independent, and the number of them is
        \begin{align*}
            1 + \Floor{\frac{d}{n}} + \Floor{\frac{d}{2n-1}} + \Floor{\frac{d - n}{2n - 1}}
            &= 1 + \Floor{\frac{d}{n}} + \Floor{\frac{d}{2n-1}} + \Floor{\frac{d - n + 1/2}{2n - 1}} \\
            &= \Floor{\frac{d}{n}} + \Floor{\frac{d}{2n-1}} + \Floor{\frac{d}{2n - 1} + \frac{1}{2}} \\
            &= \Floor{\frac{d}{n}} + \Floor{\frac{2d}{2n-1}} \\
            &= \dim (S_{D'})_{\deg = d},
        \end{align*}
        so in fact \eqref{eq:basis_S_D'} are a basis for $S_{D'}$, so there are no more generators or relations needed.
        
        The generation of the quotient spaces $S_{D''} / S_{D'}$ and $S_D / S_{D'}$ follows the genus zero case: a $\Bbbk$-basis is indexed by combinations $(d,c^{(k)})$ of degree and pole order at the respective point with
        \begin{alignat*}{2}
            \frac{2}{\ceil{2/\alpha^{(1)}}} d &< c^{(1)} \leq \alpha^{(1)} d, & \quad k &= 1\ (\text{for } S_{D''} / S_{D'}) \\
            \frac{1}{n} d &< c^{(2)} \leq \alpha^{(2)} d, & \quad k &= 2\ (\text{for } S_{D} / S_{D''}),
        \end{alignat*}
        and each $(d,c^{(k)})$ is minimally achieved by a product of consecutive generators ${f_i^{(k)}}^a {f_{i+1}^{(k)}}^b$, where $f_w$ stands in for both $f_1^{(1)}$ and $f_1^{(2)}$. Consequently, any product of two generators not of this type is the leading term of a relation, as claimed.
        
        It remains to determine whether the relation with the cubic, boxed leading term $f_2^{(1)} f_w^{2}$ (degree $4n-1$) is minimal. We divide into cases by the values of the $\{-1/\alpha^{(i)}\}$, as claimed in the theorem.
        
        If $\{-1/\alpha^{(1)}\} \in [2/3, 1)$, that is, $\alpha^{(1)} \geq 3/(3n-2)$, there is a best lower approximation $c_3^{(1)}/d_3^{(1)} = 3/(3n-2)$ giving a generator $f_3^{(1)}$ in degree $3n - 2$ with a triple pole at $P^{(1)}$. Using the relations
        \begin{align*}
            R_1 &= f_3^{(1)} u + \cdots \\
            R_2 &= f_3^{(1)} f_w + \cdots,
        \end{align*}
        we take a linear combination $R = f_w R_1 - u R_2$, causing the leading terms $f_3^{(1)} f_w u$ to cancel. Observe that $R_1$ has a term $f_2^{(1)} f_w$, the only possible monomial in degree $4n - 1$ with a triple pole at $P^{(1)}$ below $f_3^{(1)} u$ in the term ordering. So $R$ has a term $f_2^{(1)} f_w^{2}$. All other terms of $f_w R_1$ and $u R_2$, except the leading terms which cancel, are lower in the term ordering (namely $f_2^{(1)}f_w u^n, f_2^{(1)} u^{2n}, f_w u^{3n-1}$, and $u^{4n-1}$). So we have achieved a relation $R$ with the desired leading term, showing that the boxed Gr\"obner basis element is not minimal.
        
        Similarly, if $\{-1/\alpha^{(2)}\} \in [1/2, 1)$, that is, $\alpha^{(2)} \geq 2/(2n - 1)$, there is a best lower approximation $c_2^{(2)}/d_2^{(2)} = 2/(2n - 1)$ giving a generator $f_2^{(2)}$ in degree $2n - 1$ with a double pole at $P^{(2)}$. Using the relations
        \begin{align*}
            R_1 &= f_2^{(2)} u + \cdots \\
            R_2 &= f_2^{(2)} f_2^{(1)} + \cdots,
        \end{align*}
        we form a combination
        \[
        R = f_2^{(1)} R_1 - u R_2 - c u^{2n-1} R_1,
        \]
        where the first two terms have canceling leading terms $f_2^{(2)} f_2^{(1)} u$, and the constant $c$ is chosen to cancel out the term $f_2^{(2)} u^{2n}$ which may appear in $u R_2$. The highest remaining term is $f_2^{(1)} f_w^2$, which must appear with a nonzero coefficient because $f_w^2$ is the only other term in $R_1$ that can have a double pole at $P^{(1)}$. So, again, the boxed Gr\"obner basis element is not minimal.
        
        Finally, if $\{-1/\alpha^{(1)}\} \in [{1}/{2}, {2}/{3})$ and $\{-1/\alpha^{(2)}\} \in [0, {1}/{2})$, we claim that the boxed Gr\"obner basis element is minimal. If not, it arises by canceling the leading terms of other relations, so there must be a monomial in the generators of degree $4n - 1$ divisible by two different leading terms of relations. We have $\deg f_2^{(2)} \geq 3n - 1$, $\deg f_3^{(2)} \geq 4n - 2$, and $\deg f_3^{(1)} \geq 5n - 3$, so we can restrict our sights to $u$, $f_w$, $f_2^{(1)}$, and $f_2^{(2)}$. The only relations among these have leading terms $u f_2^{(2)}$, $f_w f_2^{(2)}$, and $f_2^{(1)} f_w^2$, all of which have degree at least $4n - 1$, so this is impossible.
    \end{proof}
    
    \begin{remark}
        The structure of the ring is somewhat simpler in this case and much more nearly symmetric between $\alpha^{(1)}$ and $\alpha^{(2)}$.
    \end{remark}
    
    \section{The subtle behavior of the ineffective two-point case} 
    \label{sec:unans}
    
    It would be desirable to extend the results of this paper to ineffective two-point divisors $D$ having positive multiplicity at one point and negative at the other. However, in this situation, the section ring depends on the choice of curve $C$ and divisor $D$ in a much more subtle way.
    
    \begin{example}
        Let $D = 4 P^{(1)} - P^{(2)}$. In degree $1$, we have three generators $t_2 u$, $t_3 u$, $t_4 u$. In degree $2$, the question arises of whether the six pairwise products $u^2 t_i t_j$ of the generators fill the six-dimensional space
        \[
        \left(S_D\right)_{\deg = 2} = u^2\cdot\<t_3, t_4, t_5, t_6, t_7, t_8\>.
        \]
        In fact they do. If there were a linear relation among these products, a consideration of zero and pole orders at the $P^{(k)}$ shows that it would have to have the form
        \[
        t_3^2 = c t_2 t_4, \quad c \in \Bbbk^\cross,
        \]
        but the two sides do not have the same divisor. Hence degree $1$ generates degree $2$, and in particular $t_3$ is a linear combination of the products $t_i t_j$, $2 \leq i \leq j \leq 4$. In fact, we can be more explicit:

        Taking $y = t_3$ and $x = t_2$ as coordinates, we get a Weierstrass equation of the curve
        \[
        C : y^2 + a_1 x y + a_3 y = x^3 + a_2 x^2 + a_4 x + a_6
        \]
        where $P^{(1)}$ is the point at $\infty$ and $P^{(2)}$ is the origin $(0,0)$. We must have
        \begin{itemize}
            \item $a_6 = 0$ since $(0,0)$ lies on $C$;
            \item $a_4 = 0$ since $y$ has a double zero at $(0,0)$;
            \item $a_3 \neq 0$ since $C$ is nonsingular at $(0,0)$;
            \item $a_2 \neq 0$ or else $y$ would have a triple zero at $(0,0)$.
        \end{itemize}
        Then
        \[
        t_4 = x^2 - \frac{a_3}{a_2} y,
        \]
        and the curve's equation can be written as
        \[
        t_3 = \frac{1}{a_3} \left( t_2 t_4 - \frac{a_3}{a_2} t_2 t_3 + a_2 t_2^2 - t_3^2 \right),
        \]
        showing that indeed $t_3 \in \<t_2 t_4, t_3^2, t_2 t_3, t_2^2\>$. Note that $t_3$'s pole of order $3$ at $P^{(1)}$ arises from canceling functions with poles of order $6$, $6$, $5$, and $4$.
    \end{example}
    
    \begin{example}
        \label{example: subtle 2pt behavior}
        Let $p$ and $q$ be coprime positive integers. By the Euclidean algorithm, there are positive integers $a$ and $b$ such that $aq - bp = 1$. Let $P^{(1)}$ and $P^{(2)}$ be points whose difference is not torsion, and consider the divisor
        \[
        D = \frac{a}{p} P^{(1)} - \frac{b}{q} P^{(2)}
        \]
        of degree $1/(pq)$. For $d > 0$, the dimension of $(S_D)_{\deg = d}$ is
        \[
        \dim (S_D)_{\deg = d} = \max\left\{0, \Floor{\frac{ad}{p}} - \Ceil{\frac{bd}{q}} - 1 \right\}. 
        \]
        We recognize the right side as the number of ways of writing $d - pq$ in the form $xp + yq$ where $x,y \in \bbZ_{\geq 0}$. The least degree in which there is any element is $d = pq$. For $pq \leq d < 2pq$, the degrees in which $S_D$ is nonzero are of the form $d = pq + k$ where $k$ belongs to the Frobenius (symmetric, two-generated) semigroup
        \[
        \<p, q\> = \{ xp + yq : x,y \in \bbZ_{\geq 0} \},
        \]
        and due to the degree bound $pq \leq d < 2pq$, none of these can generate each other. This yields an intricate pattern of generators and relations.
    \end{example}
    
    \subsection{A conjecture on generators}
    
    Suppose that $P^{(1)}$ and $P^{(2)}$ are two points on $C$ such that $P^{(1)}-P^{(2)}$ is not a torsion class in the Picard group (the generic case). For $c \geq 0$, let $t_c$ denote the unique function (up to scaling) on $C$ with
    \[\div(t_c^{(1)})=-c(P^{(1)})+(c-1)(P^{(2)})+(cP^{(1)}\oplus(1-c)P^{(2)}),\] where $\oplus$ denotes addition in the group law on $C$, as opposed to formal addition of divisors. Observe that if $c \geq 2$, then 
    \[\ord_{P^{(1)}}t_c = -c \textand \ord_{P^{(2)}}t_c = c-1.\]
    
    \begin{lemma}
        \label{lemma: Z-basis for ineffective 2pt case}
        Let $D=\alpha^{(1)}(P^{(1)})-\alpha^{(2)}(P^{(2)})$ be a $\bbZ$-divisor on an elliptic curve supported at two points, where $\alpha^{(1)}>\alpha^{(2)}\geq 0.$ Then $H^0(D)$ has dimension $\alpha^{(1)}-\alpha^{(2)}$ with a basis consisting of the functions
        \begin{itemize}
            \item $\{1,t_2,\ldots, t_{\alpha^{(1)}}\}$ if $\alpha^{(2)}=0$;
            \item $\{t_{\alpha^{(2)}+1},\ldots, t_{\alpha^{(1)}}\}$ if $\alpha^{(2)} > 0$.
        \end{itemize}
    \end{lemma}
    \begin{proof}
        Since $\deg D = \alpha^{(1)}-\alpha^{(2)}\geq 1$, the Riemann-Roch theorem gives us the dimension $h^0(D)$. We check that the claimed functions belong to $H^0(D)$ and have different pole orders at $P^{(1)}$, so they must form a basis. 
    \end{proof}

    \begin{figure}[ht]
        \hfill
        \begin{picture}(2.4,7)(-0.2,-1)
            \drawline(0,0)(2,0)
            \drawline(0,0)(0,6)
            
            \drawline(0,0)(1.2,6)
            
            \solid{(1,0)}\solid{(1,2)}\solid{(1,3)}\solid{(1,4)}\solid{(1,5)}
            
            \xdot{(1,1)}\xdot{(2,1)}
            
            \open{(0,0)}
            \open{(2,0)}\open{(2,2)}\open{(2,3)}\open{(2,4)}\open{(2,5)}
            
            \put(1,-1){\makebox(0,1)[b]{$5P^{(1)}$}}
        \end{picture}
        \hfill
        \begin{picture}(2.4,7)(-0.2,-1)
            \drawline(0,0)(2,0)
            \drawline(0,0)(0,6)
            
            \drawline(0,0)(1.2,6)
            \drawline(0,0)(2.2,2.2)
            
            \solid{(1,2)}\solid{(1,3)}\solid{(1,4)}\solid{(1,5)}
            
            \xdot{(1,1)}\xdot{(2,2)}
            
            \open{(2,3)}\open{(2,4)}\open{(2,5)}
            
            \put(1,-1){\makebox(0,1)[b]{$5P^{(1)}-P^{(2)}$}}
        \end{picture}\hfill
        \begin{picture}(2.4,7)(-0.2,-1)
            \drawline(0,0)(2,0)
            \drawline(0,0)(0,6)
            
            \drawline(0,0)(1.2,6)
            \drawline(0,0)(2.4,4.8)
            
            \solid{(1,3)}\solid{(1,4)}\solid{(1,5)}
            
            \xdot{(1,2)}\xdot{(2,4)}
            
            \open{(2,5)}
            
            \put(1,-1){\makebox(0,1)[b]{$5P^{(1)}-2P^{(2)}$}}
        \end{picture}\hfill
        \begin{picture}(2.4,11.5)(-0.2,-1)
            \drawline(0,0)(2.1,0)
            \drawline(0,0)(0,10.5)
            
            \drawline(0,0)(2.1,10.5)
            \drawline(0,0)(2.1,6.3)
            
            \solid{(1,4)}\solid{(1,5)}\solid{(2,7)}
            
            \xdot{(1,3)}\xdot{(2,6)}
            
            \open{(2,8)}\open{(2,9)}\open{(2,10)}
            
            \put(1,-1){\makebox(0,1)[b]{$5P^{(1)}-3P^{(2)}$}}
        \end{picture}\hfill
        \begin{picture}(3.4,16.5)(-0.2,-1)
            \drawline(0,0)(3.1,0)
            \drawline(0,0)(0,15.5)
            
            \drawline(0,0)(3.1,15.5)
            \drawline(0,0)(3.1,12.4)
            
            \solid{(1,5)}\solid{(2,9)}\solid{(3,13)}
            
            \xdot{(1,4)}\xdot{(2,8)}\xdot{(3,12)}
            
            \open{(2,10)}
            \open{(3,14)}\open{(3,15)}
            
            \put(1.5,-1){\makebox(0,1)[b]{$5P^{(1)}-4P^{(2)}$}}
        \end{picture}
        \hfill {}
        \caption{Example bases for $S_D$ with $D$ as in Lemma \ref{lemma: Z-basis for ineffective 2pt case}
        }
        \label{fig:2pt ineffective Z-basis}
    \end{figure}
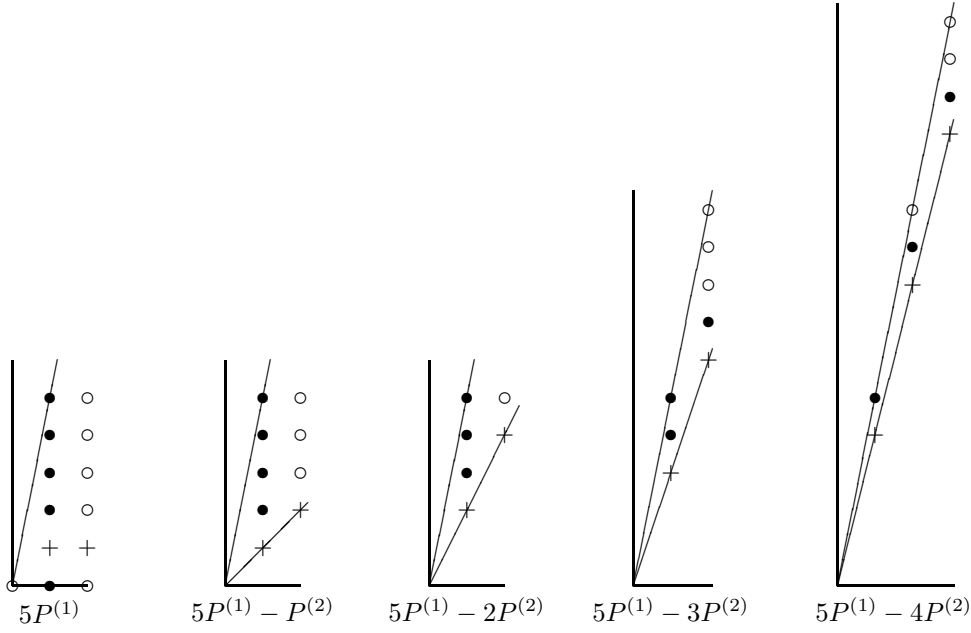
    
    Let $D=\alpha^{(1)}(P^{(1)})-\alpha^{(2)}(P^{(2)})$ be a $\bbQ$-divisor supported at the two points $P^{(k)}$, and suppose that $\alpha^{(1)} > \alpha^{(2)} > 0$ so that $S_D$ is non-trivial. Since $P^{(1)} - P^{(2)}$ is not torsion, $H^0(dD)$ does not contain any function $f$ of degree $d$ with $\ord_{P^{(1)}}(f)=\ceil{d\alpha^{(2)}}$. Accordingly, we mark the points $(d,\ceil{d\alpha^{(2)}})$ in the monoid $M'=\{(d,c)\in \bbZ^2:d\alpha^{(2)}\leq c\leq d\alpha^{(1)}\}$ by a `+' in our examples.

    \begin{conjecture}
        \label{conj:2point ineffective}
        Let $D=\alpha^{(1)}(P^{(1)})-\alpha^{(2)}(P^{(2)})$ be an ineffective ($\alpha^{(1)}>\alpha^{(2)} > 0$) $\bbQ$-divisor on an elliptic curve $C$ supported at two points $P^{(k)},$ where $P^{(1)}-P^{(2)}$ is not a torsion class in the Picard group (the generic case). Let 
        \[0=\frac{c_0^{(1)}}{d_0^{(1)}}<\frac{c_1^{(1)}}{d_1^{(1)}}<\cdots<\frac{c_r^{(1)}}{d_r^{(1)}}=\alpha^{(1)};~ 0=\frac{c_0^{(2)}}{d_0^{(2)}}>\frac{c_1^{(2)}}{d_1^{(2)}}>\cdots>\frac{c_s^{(2)}}{d_s^{(2)}}=\alpha^{(2)}\] be the best lower approximations to $\alpha^{(1)}$ and the best upper approximations to $\alpha^{(2)}$ respectively.
        Let $M=\{(d,c)\in \bbZ^2: d\alpha^{(2)} + 1 \leq c\leq d\alpha^{(1)}\}.$
        
        Then $S_D$ has a minimal system of generators of the following forms:
        \begin{enumerate}[$($a$)$]
            \item $t_{c_j^{(1)}} u^{d_j^{(1)}},$ for $j=s^{(2)}+1,\ldots, r^{(1)}$ if $c_j^{(1)}> \ceil{ d_j^{(2)}\alpha^{(2)}}$ for some $j$; 
            \item $t_{c_j^{(2)}+1} u^{d_j^{(2)}},$ for $j=2,\ldots, s^{(2)}$ if such a generator has not already appeared and no $(d_j^{(2)},c_j^{(2)}+n)\in M$ with $n\geq 2$ is (at least) two distinct nonnegative linear combinations of $(d,c)\in M$ with $d<d_j^{(2)}$; 
            \item $t_{c_j^{(2)}+2}u^{d_j^{(2)}},$ for $j=3,\ldots, s^{(2)}$ if all of the following conditions are met:
            \begin{itemize}
                \item $c_j^{(2)}+2\leq d_j^{(2)}\alpha^{(1)},$
                \item $(d_j^{(2)}, c_j^{(2)}+2)$ is not a nonnegative linear combination of any $(d,c)\in M$ with $d<d_j^{(2)},$ and
                \item no point $(d_j^{(2)},c_j^{(2)}+n)\in M$ for $n\geq 3$ is (at least) two distinct nonnegative linear combinations of $(d,c)\in M$ with $d<d_j^{(2)}.$
            \end{itemize}
        \end{enumerate}
    \end{conjecture}
    
    \begin{remark}
        Each type of generator in Conjecture \ref{conj:2point ineffective} is minimal because of how we have defined them. As in \cite{O'Dorney}, Theorems \ref{thm:1point_gens} and \ref{thm:2point_gens} and \cite{Landesman-Ruhm-Zhang-Spin-canonical-rings}, no best lower approximation to $\alpha^{(1)}$ comes from combining functions in lower degrees, and by definition no generator of type $\ib$ or $\ic$ corresponds to a linear combination of other generators, nor some difference of functions as in the proof of Theorem \ref{thm:2point_gens}. For examples of $D$ with sufficiently large degrees we can use Magma to determine the degrees of minimal generators for $S_D$ including Example \ref{example: D=2/3P^{(1)}-3/5P^{(2)}} which has the subtle behavior of Example \ref{example: subtle 2pt behavior}, and generators from Conjecture \ref{conj:2point ineffective} seem to give a basis.
        
        However, it is difficult to verify this conjecture rigorously, especially in cases such as Example \ref{example: subtle 2pt behavior} with small degree, as Magma needs to check for generators in such large degrees as to be computationally prohibitive. Even if Conjecture \ref{conj:2point ineffective} is true, the question remains of whether we can find a simpler description of the generators which does not rely on manually working out every possible linear combination of vectors in the monoid $M$ while successively adding generators in the order indicated by Definition \ref{def:term_order_1pt}.
    \end{remark}

    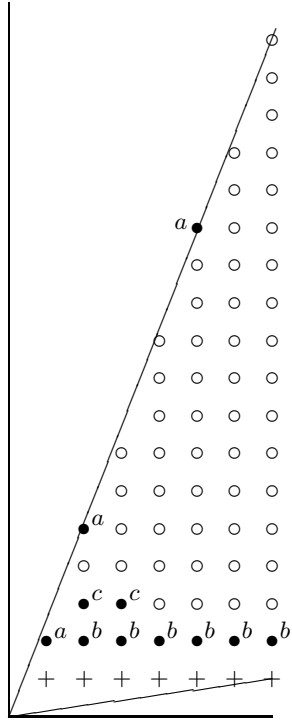
\begin{figure}[htbp]
        \centering
        \begin{picture}(7,19)(-0.5,-0.5)
            \drawline(0,0)(7,0)
            \drawline(0,0)(0,19)
            
            \drawline(0,0)(7.1,18.3)
            \drawline(0,0)(7,1)
            
            \solid{(1,2)}\solid{(2,3)}\solid{(2,5)}\solid{(5,13)}\solid{(2,2)}\solid{(3,2)}\solid{(4,2)}\solid{(5,2)}\solid{(6,2)}\solid{(7,2)}\solid{(3,3)}
            \put(1.2,2.1){\makebox(0,0)[bl]{\small $a$}} 
            \put(2.2,3.1){\makebox(0,0)[bl]{\small $c$}} 
            \put(2.2,5.1){\makebox(0,0)[bl]{\small $a$}} 
            \put(4.4,13.1){\makebox(0,0)[ul]{\small $a$}} 
            \put(2.2,2.1){\makebox(0,0)[bl]{\small $b$}} 
            \put(3.2,2.1){\makebox(0,0)[bl]{\small $b$}}
            \put(4.2,2.1){\makebox(0,0)[bl]{\small $b$}} 
            \put(5.2,2.1){\makebox(0,0)[bl]{\small $b$}}  
            \put(6.2,2.1){\makebox(0,0)[bl]{\small $b$}} 
            \put(7.2,2.1){\makebox(0,0)[bl]{\small $b$}} 
            \put(3.2,3.1){\makebox(0,0)[bl]{\small $c$}} 
            
            \xdot{(1,1)}\xdot{(2,1)}\xdot{(3,1)}\xdot{(4,1)}\xdot{(5,1)}\xdot{(6,1)}\xdot{(7,1)}
            
            \open{(4,3)}\open{(5,3)}\open{(6,3)}\open{(7,3)}
            \open{(2,4)}\open{(3,4)}\open{(4,4)}\open{(5,4)}\open{(6,4)}\open{(7,4)}
            \open{(3,5)}\open{(4,5)}\open{(5,5)}\open{(6,5)}\open{(7,5)}
            \open{(3,6)}\open{(4,6)}\open{(5,6)}\open{(6,6)}\open{(7,6)}
            \open{(3,7)}\open{(4,7)}\open{(5,7)}\open{(6,7)}\open{(7,7)}
            \open{(4,8)}\open{(5,8)}\open{(6,8)}\open{(7,8)}
            \open{(4,9)}\open{(5,9)}\open{(6,9)}\open{(7,9)}
            \open{(4,10)}\open{(5,10)}\open{(6,10)}\open{(7,10)}
            \open{(5,11)}\open{(6,11)}\open{(7,11)}
            \open{(5,12)}\open{(6,12)}\open{(7,12)}
            \open{(6,13)}\open{(7,13)}
            \open{(6,14)}\open{(7,14)}
            \open{(6,15)}\open{(7,15)}
            \open{(7,16)}
            \open{(7,17)}
            \open{(7,18)}
        \end{picture}
        \caption{Generators for $S_D$ labeled according to Conjecture \ref{conj:2point ineffective} when $D=\frac{13}{5}P^{(1)}-\frac{1}{7}P^{(2)}.$}
    \end{figure}
    
    \newpage
    \begin{example}
        \label{example: D=2/3P^{(1)}-3/5P^{(2)}}
        Let $D=\frac{2}{3}P^{(1)}-\frac{3}{5}P^{(2)}.$ This is an example of the behavior discussed in Example \ref{example: subtle 2pt behavior}. Up to degree $60,$ Magma computes generators for $S_D$ in degrees:
        \[15, 18, 20, 21, 23, 24, 25, 26, 27, 28, 29, 30, 31, 32, 33, 34, 35, 37. \]
        
        \begin{figure}[htbp]
            \scalebox{0.75}{
                \begin{picture}(37,25)(-0.5,-0.5)
                    \drawline(0,0)(37,0)
                    \drawline(0,0)(0,25)
                    
                    \drawline(0,0)(37,24.67)
                    \drawline(0,0)(37,22.2)
                    
                    \solid{(15,10)}\solid{(18,12)}\solid{(20,13)}\solid{(21,14)}\solid{(23,15)}\solid{(24,16)}\solid{(25,16)}\solid{(26,17)}\solid{(27,18)}\solid{(28,18)}\solid{(29,19)}\solid{(30,19)}\solid{(31,20)}\solid{(32,21)}\solid{(33,21)}\solid{(34,22)}\solid{(35,22)}\solid{(37,24)}
                    \put(15.1,10.0){\makebox(0,-0.3)[ul]{\small \emph a}}
                    \put(18.1,12.0){\makebox(0,-0.3)[ul]{\small \emph b}}
                    \put(20.1,13.0){\makebox(0,-0.3)[ul]{\small \emph b}}
                    \put(21.1,14.0){\makebox(0,-0.3)[ul]{\small \emph b}}
                    \put(23.1,15.0){\makebox(0,-0.3)[ul]{\small \emph b}}
                    \put(24.1,16.0){\makebox(0,-0.3)[ul]{\small \emph b}}
                    \put(25.1,16.0){\makebox(0,-0.3)[ul]{\small \emph b}}
                    \put(26.1,17.0){\makebox(0,-0.3)[ul]{\small \emph b}}
                    \put(27.1,18.0){\makebox(0,-0.3)[ul]{\small \emph b}}
                    \put(28.1,18.0){\makebox(0,-0.3)[ul]{\small \emph b}}
                    \put(29.1,19.0){\makebox(0,-0.3)[ul]{\small \emph b}}
                    \put(30.1,19.0){\makebox(0,-0.3)[ul]{\small \emph b}}
                    \put(31.1,20.0){\makebox(0,-0.3)[ul]{\small \emph b}}
                    \put(32.1,21.0){\makebox(0,-0.3)[ul]{\small \emph b}}
                    \put(33.1,21.0){\makebox(0,-0.3)[ul]{\small \emph b}}
                    \put(34.1,22.0){\makebox(0,-0.3)[ul]{\small \emph b}}
                    \put(35.1,22.0){\makebox(0,-0.3)[ul]{\small \emph b}}
                    \put(37.1,24.0){\makebox(0,-0.3)[ul]{\small \emph b}}
                    
                    \xdot{(3,2)  }\xdot{(5,3)  }\xdot{(6,4)  }\xdot{(8,5)  }\xdot{(10,6) }\xdot{(11,7) }\xdot{(12,8) }\xdot{(13,8) }\xdot{(14,9) }\xdot{(15,9) }\xdot{(16,10)}\xdot{(17,11)}\xdot{(18,11)}\xdot{(19,12)}\xdot{(20,12)}\xdot{(21,13)}\xdot{(22,14)}\xdot{(23,14)}\xdot{(24,15)}\xdot{(25,15)}\xdot{(26,16)}\xdot{(27,17)}\xdot{(28,17)}\xdot{(29,18)}\xdot{(30,18)}\xdot{(31,19)}\xdot{(32,20)}\xdot{(33,20)}\xdot{(34,21)}\xdot{(35,21)}\xdot{(36,22)}\xdot{(37,23)}
                    
                    \open{(30,20)}\open{(33,22)}\open{(35,23)}\open{(36,23)}\open{(36,24)}
            \end{picture}}	
            \caption{Generators for $S_D$ labeled according to Conjecture \ref{conj:2point ineffective} where $D=\frac{2}{3}P^{(1)}-\frac{3}{5}P^{(2)}.$}
        \end{figure}
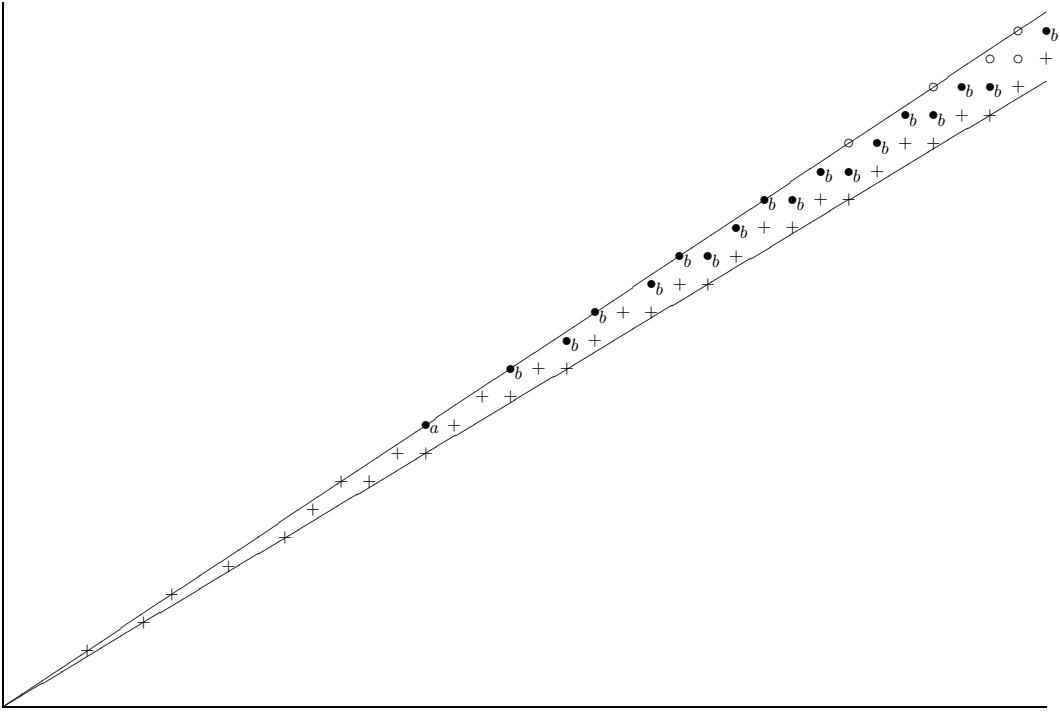 
    \end{example}   
    
    \section{Arbitrary effective \texorpdfstring{$\bbQ$}{Q}-divisors}
    
    For divisors supported by more than two points, generators and relations often occur in high degrees, and it is difficult to explicitly describe the canonical ring. Thanks to \cite{Landesman-Ruhm-Zhang-Spin-canonical-rings} we are able to determine inductive presentations of such rings for effective $\bbQ$-divisors, similar to the main inductive theorem in Voight--Zureick-Brown \cite[$8.3.1$]{VZB}.
    
    
    
    \begin{example} \label{ex:3point}
        As in \cite[Example $5.7.7$]{VZB} let $D'=\frac{1}{2}P_1+\frac{1}{2}P_2$, and following \cite[Example $5.7.9$]{VZB} let $D=D'+\frac{1}{2}P_3=\frac{1}{2}P_1+\frac{1}{2}P_2+\frac{1}{2}P_3$. Then $\deg D = 3/2$. By the Generalized Max Noether Theorem \cite[Lemma 3.1.4]{VZB}, $H^0(C,2D)\otimes H^0(C,(d-2)D)\to H^0(C,dD)$ is surjective for $d>5$, so all generators occur in degree $<5$.
        
        More precisely, in \cite{VZB} it is computed that $S_{D'}$ is generated in degrees $1$, $2$, and $4$ while $S_D$ is generated in degrees $1$, $2$, and $2$. The square of the last degree-$2$ generator of $S_D$ is the degree-$4$ generator for $S_{D'}.$
        
        So the minimal presentations have the form $S_D = \Bbbk[u,x_1,x_2]/I_D$ and $S_{D'} = \Bbbk[u,x_1,x_2^2]/I_{D'}$, where $I_D$, $I_{D'}$ are the relation ideals. In particular, $S_D$ is generated over $S_{D'}$ by $x_2$.
    \end{example}
    
    A powerful result which allows one to compute an inductive presentation of the section ring of a general $\bbQ$-divisor on an elliptic curve is \cite[Lemma $4.4$]{Landesman-Ruhm-Zhang-Spin-canonical-rings}. We paraphrase the result, which is of independent interest, in the terminology of this document for reference. We then use it to prove Theorem \ref{thm:bounds_intro} by verifying that our general $\bbQ$-divisors satisfy the hypotheses of the lemma. 

    If $D$ is a divisor on a curve $C$, $P$ is a point on $C$, and $f$ is a rational function on $C$, we define, following \cite{Landesman-Ruhm-Zhang-Spin-canonical-rings},
    \[
      \ord_P^D(f) = \ord_P(f) + \ord_P(D),
    \]
    so that $f \in H^0(D)$ if and only if $\ord_P^D(f) \geq 0$ for all points $P$.
    
    \begin{lemma}[{\cite[Lemma $4.4$]{Landesman-Ruhm-Zhang-Spin-canonical-rings}}]
        \label{l: LRZ 4.4}
        Let $C$ be a curve (of any genus) and let $D'$ be an effective $\bbQ$-divisor on $C$. Suppose that $P$ is not a basepoint of $dD'$ for any $d\in \bbN,$ i.e.\ we can choose generators $u = f_0, f_1,\ldots, f_m$ of $S_{D'}$ with $\deg(u)=1,$ and $\ord_P^{D'}(f_i)=0$ for $0\leq i\leq m$.
        
        Suppose that $\displaystyle{D=D'+\frac{a}{b}P}$ for some $a,$ $b\in \bbN$ such that $\frac{a}{b}$ is reduced and 
        \begin{equation} \label{eq:LRZ h0}
            h^0(C,\floor{dD})=h^0(C,\floor{dD'})+\Big\lfloor d \cdot \frac{a}{b}\Big\rfloor\quad\text{ for all }d\in \bbN.
        \end{equation}
        Then 
        \begin{enumerate}[$($a$)$]
            \item $S_D$ is generated over $S_{D'}$ by some elements $g_1,\ldots,g_n$ whose degrees $d_i = \deg(g_i)$ and pole orders $c_i = -{\ord}_P^{D'}(g_i)$ satisfy $c_i\leq c_{i+1}\leq a$ and $d_i\leq d_{i+1}\leq b$ for all $i$.
            
            \item Choose a monomial ordering $\prec$ on $\Bbbk[u = f_0, f_1, \ldots, f_m]$ such that \[\ord_u(f) <\ord_u(h)\Rightarrow f\prec h.\] Equip $\Bbbk[f_0, \ldots, f_m]$ with the graded $P$-lexicographic order from \cite[Definition $4.2$]{Landesman-Ruhm-Zhang-Spin-canonical-rings} and equip 
            $\displaystyle{\Bbbk[g_1,\ldots, g_n]\otimes \Bbbk[f_0, \ldots, f_m]}$ with the block order from \cite[Definition $2.19$]{Landesman-Ruhm-Zhang-Spin-canonical-rings}. Let $I'$ denote the ideal of relations of 
            \[\Bbbk[f_0,\ldots, f_m]\to S_{D'}\] and let $I$ denote the ideal of relations of 
            \[\Bbbk[f_0,\ldots, f_m,g_1,\ldots, g_n]\to S_D.\] Then 
            \[\operatorname{in}_{\prec}(I)=\operatorname{in}_{\prec}(I')\Bbbk[f_0,\ldots, f_m,g_1,\ldots, g_n] + \langle U_i: 1\leq i\leq n \rangle + \langle V \rangle,\]
            where $\displaystyle{V=\{f_ig_j: 1\leq i\leq m, 1\leq j\leq n \}}$ and $U_i$ is the set of monomials of the form $\prod_{j=1}^i g_j^{e_j}$ with $e_j\in \bbN_{\geq 0}$ such that 
            \begin{enumerate}
                \item $\sum_{j=1}^i e_jc_j\leq c_{i+1},$
                \item there does not exist $(e'_1,\ldots,e'_i)\neq (e_1,\ldots, e_i)$ with all $e'_j\leq e_j$ and $\sum_{j=1}^i e'_jc_j\geq c_{i+1},$ and 
                \item there does not exist some $r<i$ such that $\sum_{j=1}^r e_jc_j>c_{r+1}.$
            \end{enumerate}
            
            \item Let $\tau = \max_i \deg f_i$. Then $S_D$ is generated over $S_{D'}$ in degrees up to $b$, with $I$ generated over $I'$ in degrees up to $\max\{2b,b+\tau\}$.
        \end{enumerate}
    \end{lemma}
    \begin{remark}
        Note that the condition $h^0(C,\floor{D'})\geq 1$ from the original statement of \cite[Lemma $4.4$]{Landesman-Ruhm-Zhang-Spin-canonical-rings} is automatic any for effective $\bbQ$-divisor $D$ on a genus $1$ curve $C$ since we have $h^0(C,\floor{D})=\max\{\deg \floor{D},1\}$. Also, for $u$ we can take the usual $u$ in the definition of the section ring.
    \end{remark}

    \begin{proof}[Proof of Theorem \ref{thm:bounds_intro}]
      Now let
      \[
        D = \sum_{i=1}^n \frac{a^{(i)}}{b^{(i)}} \big(P^{(i)}\big)
      \]
      be an effective divisor, with $\alpha^{(1)} \geq \cdots \geq \alpha^{(n)}$. We prove that the section ring $S_D$ is generated in degrees at most
      \[
        B = \max\{3b^{(1)}, b^{(2)}, \ldots, b^{(n)}\},
      \]
      with relations in degrees at most $2B$.

      In the base case $n = 1$, we are claiming that the section ring $S_D$ of a divisor $D = (a/b)P$  is generated in degrees at most $B = 3 b$ with relations in degrees at most $6 b$. The generator bound follows from Theorem \ref{thm:1point_gens}, observing that the exceptional generator $\ic$ has degree at most $\ceil{3b/a} \leq 3b$. The relation bound is automatic for relations with quadratic leading terms. By Theorem \ref{thm:1point_rels}, the only other minimal relation has leading term $f_0 f_3^2$ and degree $1 + 2 d_3 \leq 3b < 6b$, completing the proof of the base case.

      To prove the induction step, we must verify that the subdivisor
      \[
        D' = \sum_{i = 1}^{n - 1} \frac{a^{(i)}}{b^{(i)}}\big(P^{(i)}\big)
      \]
      and the point $P = P^{(n)}$ satisfy the hypotheses of Lemma \ref{l: LRZ 4.4}. That $P$ is not a basepoint of $dD'$ is automatic for us, because either
      \begin{itemize}
          \item $\floor{dD'} = 0$, and the constant $1 \in H^0(dD')$ has no basepoints, or
          \item $\deg \floor{dD'} = 1$, and the constant $1 \in H^0(dD')$ has a basepoint $P_i \neq P$, or
          \item $\deg \floor{dD'} \geq 2$, and the linear system $H^0(dD')$ is basepoint-free by Fact \ref{fact: principal divisors on elliptic curves}.
      \end{itemize}
      As to \eqref{eq:LRZ h0}, the hypothesis $\alpha_1 \geq \cdots \geq \alpha_n$ ensures that $\floor{dD} = 0$ exactly when $d < \ceil{1/\alpha_1}$. If this condition holds, then \eqref{eq:LRZ h0} reduces to $1 = 1 + 0$, which is true. Otherwise, \eqref{eq:LRZ h0} reduces to
      \[
        \deg \floor{dD} - 1 = (\deg \floor{dD'} - 1) + \deg \Floor{d \cdot \frac{a}{b} \cdot P},
      \]
      which is also true. Thus Lemma \ref{l: LRZ 4.4} applies.
      
      By induction, $S_{D'}$ is generated in degrees at most
      \[
        B' = \max\{3b^{(1)}, b^{(2)}, \ldots, b^{({n-1})}\},
      \]
      with relations in degrees at most $2B'$. Accordingly, $S_D$ is generated over $S_{D'}$ in degrees at most $b^{(n)}$, so generated over $\Bbbk$ in degrees at most
      \[
        \max\{B', b^{(n)}\} = B.
      \]
      The relation ideal $I$ is generated over its counterpart $I'$ in degrees at most
      \[
        \max\{2b^{(n)}, b^{(n)} + \tau\} \leq \max\{2b^{(n)}, b^{(n)} + B'\},
      \]
      and since $I'$ is generated in degrees at most $2B'$, the degrees of all relations are bounded by
      \[
        \max\{2b^{(n)}, b^{(n)} + B', 2B'\} = 2B,
      \]
      as desired.
    \end{proof}

    \bibliographystyle{amsalpha}
    \bibliography{bibliography} 
\end{document}